\documentclass[letterpaper]{article}
\usepackage{lineno}
\modulolinenumbers[5]
\usepackage[sort&compress]{natbib}
\usepackage{moreverb,url,authblk}
\usepackage[margin=1in]{geometry}
\usepackage[colorlinks,bookmarksopen,bookmarksnumbered,citecolor=red,urlcolor=red]{hyperref}

\newcommand\BibTeX{{\rmfamily B\kern-.05em \textsc{i\kern-.025em b}\kern-.08em
T\kern-.1667em\lower.7ex\hbox{E}\kern-.125emX}}

\usepackage{booktabs,tabularx,makecell,colortbl,multirow}
\newcolumntype{C}{>{\centering\arraybackslash}X}
\newcolumntype{L}{>{\raggedright\arraybackslash}X}
\newcolumntype{R}{>{\raggedleft\arraybackslash}X}

\usepackage{textcomp}

\usepackage[linesnumbered,vlined,ruled]{algorithm2e}
\usepackage{geometry,xcolor,tabularx,nccmath}
\usepackage{amsmath,amsthm,amssymb,enumitem,bm}
\usepackage{booktabs,caption,mathtools,multirow,subfig}
\SetKwComment{Comment}{/* }{ */}

\DeclareMathAlphabet{\mathdutchcal}{U}{dutchcal}{m}{n}
\SetMathAlphabet{\mathdutchcal}{bold}{U}{dutchcal}{b}{n}
\DeclareMathAlphabet{\mathdutchbcal}{U}{dutchcal}{b}{n}
\newcommand{\sA}{\mathcal{A}}

\newcommand{\sK}{\mathcal{K}}

\newcommand{\sG}{\mathcal{G}}
\newcommand{\sV}{\mathcal{V}}
\newcommand{\sR}{\mathcal{R}}
\newcommand{\sC}{\mathcal{C}}

\newcommand{\sP}{\mathcal{P}}
\usepackage[nameinlink]{cleveref}
\hypersetup{colorlinks={true},allcolors={blue}}
\crefname{constraint}{constraint}{constraints}
\crefformat{constraint}{#2constraint~\textup{(#1)}#3}
\crefrangeformat{constraint}{constraints~(#3#1#4)--(#5#2#6)}
\labelcrefformat{constraint}{\textup{#2(#1)#3}}

\expandafter\def\expandafter\UrlBreaks\expandafter{\UrlBreaks%
  \do\a\do\b\do\c\do\d\do\e\do\f\do\g\do\h\do\i\do\j%
  \do\k\do\l\do\m\do\n\do\o\do\p\do\q\do\r\do\s\do\t%
  \do\u\do\v\do\w\do\x\do\y\do\z\do\A\do\B\do\C\do\D%
  \do\E\do\F\do\G\do\H\do\I\do\J\do\K\do\L\do\M\do\N%
  \do\O\do\P\do\Q\do\R\do\S\do\T\do\U\do\V\do\W\do\X%
  \do\Y\do\Z}

\DeclareMathOperator*{\argmin}{arg\,min}

\newcommand{\cconv}{c_0}
\newcommand{\ccng}{c_\mathrm{g}}
\newcommand{\cev}{c_\mathrm{e}}
\newcommand{\ilsparam}[1]{\theta^\mathrm{#1}}
\newcommand{\bestsol}{(R^B, \kappa^B)}
\newcommand{\currentsol}{(R, \kappa)}
\newcommand{\tempsol}{(R', \kappa')}
\allowdisplaybreaks
\newtheorem{proposition}{Proposition}

\usepackage[normalem]{ulem} %

\begin{document}

\title{Joint Routing of Conventional and Range-Extended Electric Vehicles in a Large Metropolitan Network}

\author[1]{Anirudh Subramanyam\thanks{Corresponding author: \texttt{asubramanyam@anl.gov}}}

\author[1,2]{Taner Cokyasar}

\author[1]{Jeffrey Larson}

\author[1]{Monique Stinson}

\affil[1]{Argonne National Laboratory, 9700 S. Cass Ave., Lemont, IL 60439, USA}
\affil[2]{Tarsus University, Takbas mah., Kartaltepe sk., Tarsus, Mersin 33400, Turkey}

\maketitle
\begin{abstract}
    Range-extended electric vehicles combine the higher efficiency and environmental benefits of battery-powered electric motors with the longer mileage and autonomy of conventional internal combustion engines. This combination is particularly advantageous for time-constrained delivery routing in dense urban areas, where battery recharging along routes can be too time-consuming to economically justify the use of all-electric vehicles. However, switching from electric to conventional fossil fuel modes also results in higher costs and emissions and lower efficiency. This paper analyzes this heterogeneous vehicle routing problem and describes two solution methods: an exact branch-price-and-cut algorithm and an iterated tabu search metaheuristic. From a methodological perspective, we find that the exact algorithm consistently obtains tight lower bounds that also serve to certify the metaheuristic solutions as near-optimal. From a policy standpoint, we examine a large-scale real-world case study concerning parcel deliveries in the Chicago metropolitan area and quantify various operational metrics including energy costs and vehicle miles traveled. We find that by deploying roughly 20\% of range-extended vehicles with a modest all-electric range of 33 miles, parcel distributors can save energy costs by up to 17\% while incurring less than 0.5\% increase in vehicle miles traveled. Increasing the range to 60 miles further reduces costs by only 4\%, which can alternatively be achieved by decreasing the average service time by 1~minute or increasing driver working time by 1~hour. Our study reveals several key areas of improvement on which vehicle manufacturers, distributors, and policy makers can focus their attention.
    
\end{abstract}

\section{Introduction}\label{sec:introduction}
The vehicle routing problem (VRP) and its many variants have been widely studied in the field of transportation science and logistics. %
Yet, the inherent combinatorial complexity of such problems, along with the emergence of new vehicles with distinct characteristics---such as electric vehicles (EVs), plug-in hybrid EVs (PHEVs), autonomous vehicles, and unmanned aerial vehicles---make VRPs challenging to solve. A common characteristic underlying virtually all VRP variants is the determination of optimal routes for a fleet of vehicles 
to serve transportation requests distributed across geographic locations. Although the emergence of new vehicle technologies does not change this requirement, such technologies can introduce additional complicating constraints that require new modeling approaches.

This paper is motivated by the growing federal and private impetus~\citep{biden2021,bogage} for the adoption of range-extended EVs (REEVs) and PHEVs as delivery trucks for urban transportation, including e-commerce and postal operations.
REEVs are hybrid vehicles that extend the limited driving range of battery-electric vehicles (BEVs) using a fuel-based power unit (i.e., an internal combustion engine).
This extended range can be achieved either via an electric generator that charges the battery once it reaches a minimum charge state (``charge-sustaining'' mode) after the vehicle has exhausted its all-electric range (``charge-depleting'' mode) or by constantly supplementing the battery energy (``blended'' mode).
Irrespective of the mode, the crucial point is that because the batteries do not require external recharging, REEVs combine the efficiency and environmental benefits of BEVs with the range and autonomy of conventional vehicles (CVs).
That is, REEVs can alleviate the ``range anxiety'' that can accompany pure BEVs.

Our motivation stems from the potential use of REEVs for distribution logistics. %
The Clean Energy Research Center - Truck Research Utilizing Collaborative Knowledge (CERC TRUCK) is a U.S.~Department of Energy--sponsored consortium between industry and academia that is developing a compressed natural gas (CNG) range extender medium-duty vehicle for the pickup-and-delivery market, with the goal of increasing freight efficiency by 50\% over a 2016 baseline~\citep{cerc}.
Other CNG REEVs will be on the roads shortly: a new Class-6 REEV delivery truck has been developed by Efficient Drivetrains, Inc., that runs on a CNG-series hybrid-electric powertrain~\citep{ferchau,guanetti2016energy,jayanthan}, and 
a Class-8 CNG-powered hybrid-electric commercial truck has been developed by Hyliion Inc.~with planned volume shipments in 2022~\citep{hyliion}.
Considering that CNG is environmentally cleaner and cheaper than both gasoline and diesel~\citep{yadav2019advanced,jgoreham} but has an equally robust refueling infrastructure, replacing conventional parcel delivery trucks with such CNG-based REEVs can simultaneously decrease environmental impacts and operational costs.
In fact, these reductions may be achieved by using diesel-based REEVs as well; for example, it has been shown that hybrid conversion of light-duty diesel delivery trucks can already improve fuel economy by more than 24\%~\citep{byun2021effects}.
While much of the discussion in this manuscript centers on CNG-based REEVs, the developed methods are naturally applicable to REEVs independent of their conventional fuel source.

In this study we investigate the impact of REEVs on large-scale delivery operations in terms of multiple performance metrics, including energy cost, fleet size, vehicle miles traveled (VMT), and vehicle hours traveled (VHT).
We consider various realistic scenarios where a mix of CVs and REEVs are available for delivery operations. 
We define and solve a new VRP variant in these scenarios that simultaneously determines the optimal fleet composition and truck routes under various operational constraints, including limited EV ranges, truck capacities, and driver hours-of-work limits.
This range-extended electric vehicle routing problem (REEVRP) is an instance of the heterogeneous VRP (HVRP), where the goal is to determine the optimal fleet composition and vehicle routes for a heterogeneous fleet of vehicles with different capacities and costs.
Although the REEVRP also features a heterogeneous fleet of CVs and REEVs, what distinguishes the REEVRP from a classical HVRP is the heterogeneity within the REEV vehicle type.
That is, REEVs feature two distinct propulsion modes: a conventional fuel mode and a (limited-range) battery-powered EV mode.\footnote{We use the term ``EV mode'' to refer to either the charge-depleting mode or the mileage-equivalent blended mode, depending on the hybrid vehicle technology.}
The cost of the EV mode (i.e., the electricity cost to charge batteries) is considerably lower than that of the conventional fuel mode (i.e., the cost of CNG or gasoline).
Consequently, the total cost of an REEV route is no longer a linear function of the corresponding travel time or distance, and modeling REEVs as a single vehicle type within a classical HVRP model can lead to a poor approximation of the actual costs.
In addition to this nonlinear cost structure, the REEVRP we consider captures other realistic constraints, including multiple vehicle types, vehicle capacities, driving time limits, and asymmetric travel times and vehicle speeds.

The contributions of this paper are twofold.
First, from a methodological standpoint, we formally define the REEVRP as a new VRP variant and propose two algorithms for its solution: \emph{(i)} an exact branch-price-and-cut (BPC) algorithm based on a novel set partitioning integer programming (IP) formulation and \emph{(ii)} an iterated tabu search (ITS) metaheuristic that can provide solutions for large-scale problem instances with more than 3,000 delivery locations. %
The BPC algorithm consistently certifies the ITS solutions as near-optimal for small- to medium-scale instances with up to 100 locations.
Second, from a practical standpoint, we apply these methods to analyze the impacts of REEVs on the Chicago metropolitan area under various realistic operational scenarios where we vary the proportions of CVs and EVs, EV ranges, hours-of-work limits, service times, and vehicle capacities.
We provide important managerial insights by quantifying systemwide performance metrics under these scenarios to identify the key parameters that can guide policy and decision-making.
We find the following:
\begin{itemize}
    \item Distributors can expect energy cost savings of up to 17\% while incurring less than a 0.5\% increase in VMT and VHT by deploying a small number of REEVs per depot (roughly 20\% of the fleet) each with a modest EV range of 33 miles.
    \item The marginal cost savings diminish rapidly with increasing EV range; for example, doubling the range to 60 miles further reduces costs by only 4\%.
    The reason is partly because delivery routes are time-constrained. Indeed, decreasing the average service time by 1~minute and increasing driver working time by 1~hour can decrease both the energy costs and VMT by roughly 8\% and 4\%, respectively.
    For the same reason, increasing vehicle capacity is unlikely to provide cost or VMT reductions (for the scenarios we consider).
\end{itemize}
These findings along with key operational statistics reveal potential areas of improvement where vehicle manufacturers, logistics service providers, and policy makers can focus their attention.

The outline of the manuscript is as follows. 
\Cref{sec:litrev} provides a comprehensive literature review of similar studies from a methodological perspective. \Cref{sec:description} describes the REEVRP setting, provides underlying assumptions, and briefly mentions data sources utilized. \Cref{sec:model} formulates a mathematical IP model to solve REEVRP instances and describes an exact BPC method. \Cref{sec:heuristic} illustrates the ITS metaheuristic for solving large-scale REEVRP instances. \Cref{sec:experiments} provides case studies of the computational performance of the solution methodologies; these studies also provide managerial insights for various deployment scenarios. \Cref{sec:conclusions} concludes with a discussion of the strengths and weaknesses of our study and provides potential future research directions.

\section{Literature Review}\label{sec:litrev}

The VRP literature is rich; see \citet{TothVigo:2ndEdition} for a broad overview.
The REEVRP belongs to the subclass of VRPs featuring multiple vehicle types, collectively known as heterogeneous VRPs, of which comprehensive reviews can be found in \citet{Hoff2010:industrial_aspects_of_hvrp,Irnich2014,Koc2016:thirty_years_of_HVRP}.
A distinguishing feature of the REEVRP is the use of hybrid or electric vehicle types. VRP variants that incorporate EVs have burgeoned over the past decade; see \citet{pelletier201650th,erdelic2019survey} for recent surveys.
For brevity, we discuss only similarities and differences of our work with the most directly relevant literature.

The unique features of the relevant EV-related works in the literature are the limited battery capacities and driving ranges of pure BEVs.
This limitation requires optimizing over decisions corresponding to recharging station visits in addition to standard routing decisions.
Seminal papers that first considered this additional decision layer include those of \citet{conrad2011recharging,erdougan2012green} and \citet{schneider2014electric}.
Since then, various heuristic and exact solution approaches have been proposed to more realistically model the real-world routing of EVs, including more realistic charging models (e.g., partial recharging) or models where the battery charge state is a nonlinear function of the charging time~\citep{felipe2014heuristic,montoya2016multi,desaulniers2016exact,keskin2016partial,froger2019improved}).
Extensions to heterogeneous fleets of BEVs and CVs have also been studied \citep{juan2014routing,sassi2015vehicle,goeke2015routing,lebeau2015conventional,hiermann2016electric}.

The distinguishing aspect of our work from pure BEV routing problems is the use of range extenders to obviate the need to plan for time-consuming visits to recharging stations.
Pure BEV routing problems where battery recharging is disallowed en route, such as in dense urban environments, have been studied by \citet{pelletier2018charge} and \citet{florio2021routing}.
These papers consider homogeneous fleets of BEVs and include detailed models for battery charge scheduling (at the depot) and also ensure that battery capacity is not exceeded during delivery operations.
However, the absence of alternative propulsion modes in the same vehicle leads to a simple linear cost function for each vehicle, be it the total energy cost or a distance-proportional travel cost.

The literature on REEV and PHEV routing is relatively sparse.
One of the earliest works is that of \citet{abdallah2013}, where the problem is to minimize the cost of using conventional fuels in a homogeneous fleet of PHEVs that are used to serve customers in predefined time windows.
A similar problem is considered by \citet{mancini2017hybrid}, where the  author presents a heuristic approach to route a homogeneous fleet of PHEVs that can switch between EV and conventional fuel modes at any point in the route.  
These models were generalized by \citet{hiermann2019routing}, who consider a heterogeneous fleet composed of pure BEVs, PHEVs, and CVs, with the objective of finding the optimal fleet composition that minimizes the sum of fixed cost (per used vehicle) and energy costs (for electricity and conventional fuels).
All of these studies allow for battery recharging decisions en route.
    Whereas \citet{abdallah2013} allows PHEVs the option to be recharged at potentially every delivery location, \citet{hiermann2019routing}  and  \citet{mancini2017hybrid} allow PHEVs to be recharged only at specific recharging stations.
\citet{vincent2017simulated}  and \citet{li2020svnd} consider a similar variant with homogeneous PHEVs without customer time windows but allow PHEVs to also refuel gasoline (or other conventional fuel) at specific stations.
All of these papers propose metaheuristic solution approaches that are evaluated on synthetic test cases.
In particular, \citet{abdallah2013} proposes a tabu search approach; \citet{hiermann2019routing} design a sophisticated layered metaheuristic consisting of a genetic algorithm and large neighborhood search hybridized with integer programming; \citet{mancini2017hybrid} design a large neighborhood search metaheuristic; and \citet{vincent2017simulated} and \citet{li2020svnd} propose simulated annealing and variable neighborhood descent methods, respectively.
These studies consider instances of up to 100 locations and report solution times on the order of a few hours~\citep{abdallah2013}, 10~minutes~\citep{hiermann2019routing,mancini2017hybrid}, and 1~minute~\citep{vincent2017simulated,li2020svnd}.

In contrast to the aforementioned papers, our work does not consider en route
battery recharging decisions. In other words, we assume that all vehicles
charge only at the depot and that they begin each route with a full charge.
This assumption allows us to design exact methods that can quickly solve
problems with up to 100~locations and scalable metaheuristic approaches that
can approximately solve real-world test cases consisting of more than
3,000~locations in less than an hour.
Along these lines, problems involving hybrid electric vehicles that do not
involve recharging decisions have also been studied
by \citet{doppstadt2016hybrid}. However, they  consider routing only a single
vehicle that must decide which of four different propulsion modes to use
on every single arc of its route, with each mode associated with a different
cost on each arc.
In a similar vein, models that account for the optimal power management of EV and conventional fuel modes---as a function of the travel speed, power requirements, or other arc-level parameters---with varying levels of complexity can be found in the work of \citet{bahrami2020plugin,de2021general}, and \citet{casparia2021optimal}.
Although more accurate, these models can be  complex (involving mixed-integer nonlinear optimization or nested dynamic programs),  often leading to poor scalability on real-world problems.
This level of detail cannot address the scale of networks we consider in our study; instead, we make simplifying assumptions about the cost of using alternative propulsion modes that allow the solution of real-scale instances without losing key managerial and policy insights.

\section{Problem Description}\label{sec:description}
A growing number of logistics service providers (LSPs) are considering the
adoption of REEVs into their distribution fleets to reduce their energy costs
and environmental footprint. LSPs are particularly interested in understanding
the extent of cost savings (if any) under various deployment scenarios. The
magnitude of such savings could be quantified with a reasonably accurate VRP
model that captures the essential real-world features. 

To that end, we focus our attention on the Chicago metropolitan area, where we assume that four large LSPs---Amazon, FedEx, UPS, and USPS---are tasked with delivering packages on a typical day.
We acquire demand and road network data from POLARIS, the Planning and Operations Language for Agent-based Regional Integrated Simulation, developed at Argonne National Laboratory~\citep{auld2016polaris}.
We also identify 53 depot locations of the four LSPs from publicly available sources. 
We assume the LSPs uniformly share the total number of delivery tasks between them (totaling roughly 600,000 customer locations) and that they each solve an assignment problem to determine their depot-to-customer allocations.
\Cref{fig:chicago} illustrates the resulting regional-level problem layout.
Further details of the case study and experimental design are deferred to \Cref{sec:experiments}.

\begin{figure}[!htbp]
    \begin{center}
        \includegraphics[width=0.5\linewidth]{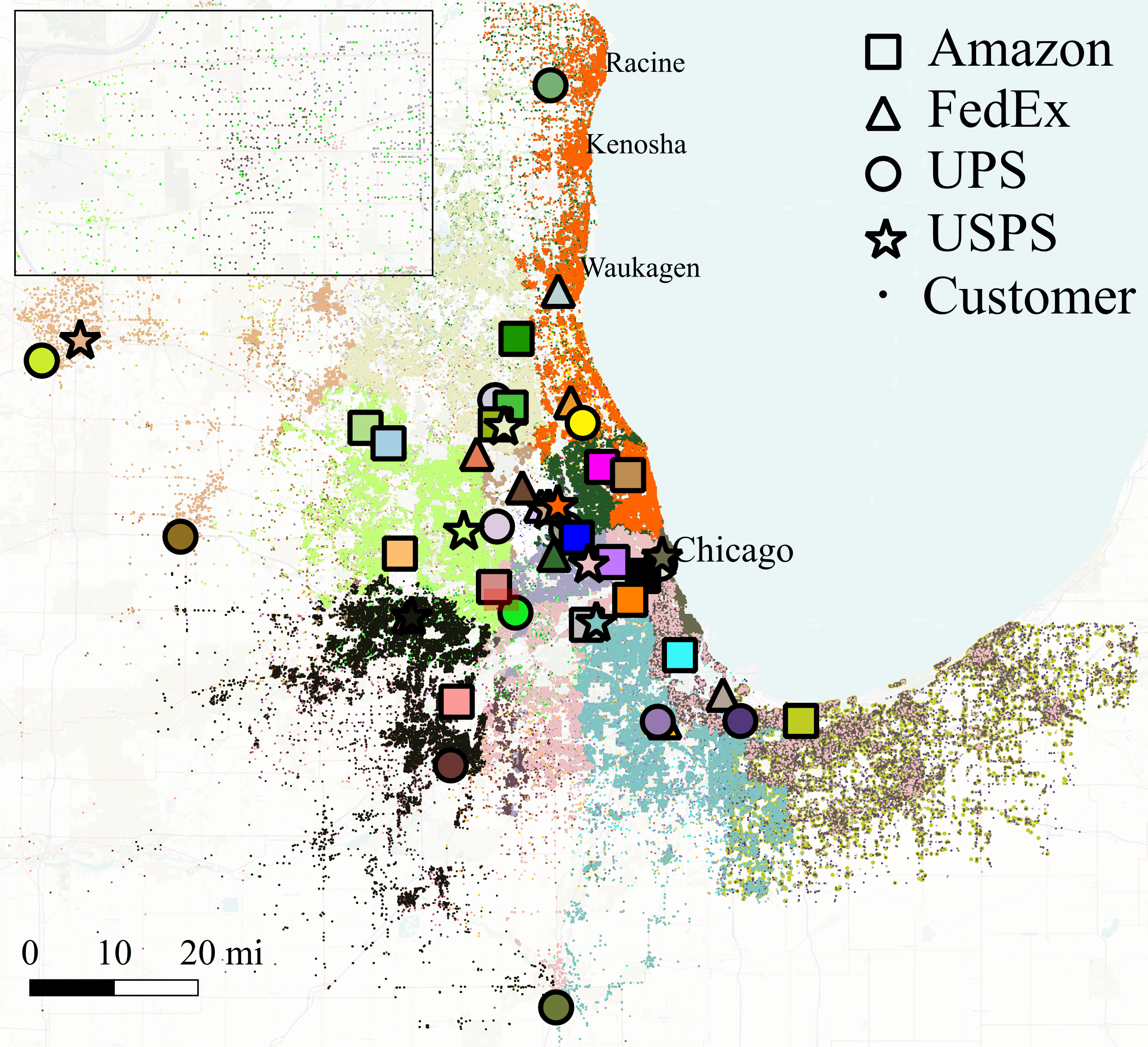}
    \end{center}
    \caption{Illustrative problem layout in the Chicago metropolitan area.
        Customer locations are color-coded to match the colors of their assigned
        depots. Although certain colors appear to dominate the main figure, this is due to overlapping points. The inset shows that customers of all LSPs appear throughout the metropolitan area. \label{fig:chicago}}
\end{figure}

With the regional-level operation decomposed into depot-level operations, we then define and solve the proposed REEVRP for each depot.
Although one can also define and solve a multidepot variant of the REEVRP for each LSP (where the depot-to-customer assignment is part of the optimization), the scale of real-world problems requires a decomposition approach, and one can view the REEVRP as a subproblem as part of this overall decomposition strategy.
Moreover, even if we focus our attention on a single depot, the considered area in \Cref{fig:chicago} already leads to fairly large REEVRP instances.

The goal of each depot-level REEVRP is to solve a combined fleet composition and routing problem.
Each depot must determine the optimal delivery routes for its trucks while simultaneously choosing a truck type (either CV or REEV) for each route.
The number of delivery trucks of each type is limited.
All vehicles are assumed to leave their depot at the beginning of the day, visit customer locations to make deliveries, and then return to their original depot.
Individual arc lengths and their speeds are provided by POLARIS; these values define the average travel times in the network.
Note that we assume the same average speed for all delivery trucks (on a given arc),  although our methodology does not require this assumption.
The REEV trucks are assumed to leave the depot with fully charged batteries and do not stop mid-route to recharge.
We assume that REEVs can optimally switch between CNG and EV modes at any point along their route and this internal power management is decided by some lower-level logic (e.g., see \citet{bahrami2020plugin,chau2017drive}), which is outside the scope of our study. 
Crucially, this implies that their EV range is always used to full capacity (assuming that their routes are longer than the EV range). 
Once the EV range has been traveled (i.e., the battery is depleted), the CNG-powered internal combustion engine (or some other conventional fuel) kicks in and allows the truck to complete the remainder of its service.

The overall objective of each depot-level REEVRP is to minimize the total energy costs.
The energy cost of CVs is readily estimated based on their miles per gallon energy efficiency and the average gallon cost of fuel (diesel in our case),
REEVs incur two types of energy costs depending on the mode (EV or CNG) by which they are powered.
The energy cost in EV mode is estimated by using the average electricity price and electricity consumption in that mode, whereas the energy cost in CNG mode is estimated based on the fuel economy in that mode and the average cost of CNG, similar to CVs.
We do not consider fixed costs per vehicle usage (e.g., to account for ownership, maintenance, or infrastructure-related costs) because of the difficulty in accurately estimating these parameters.
We note, however, that the proposed optimization methods do not preclude us from considering such fixed costs in addition to the energy-based variable costs (e.g., by simply adding them to the depot-incident arc travel costs). 

All delivery trucks are constrained by both time (i.e., work hours) and physical capacity. This means that the total time corresponding to any route (including road travel and service times) cannot exceed a prespecified number of work hours and that the total number of packages 
delivered along any route is capped by an upper limit. 
We note that the time spent for parking trucks, unloading packages, and physically making deliveries are all included in the service time.
For simplicity, we assume that all trucks have the same physical carrying capacity and work hours; as before, our models can be readily extended to lift these assumptions. 
In the following sections we describe the problem formally and provide solution methods to a general form of the problem.
We emphasize that the proposed methodology applies to any mixed REEV/CV fleet and is not tied to our motivating example of CNG-powered REEVs. 

\section{Mathematical Model and Exact Branch-Price-and-Cut Algorithm}\label{sec:model}

We first provide a precise definition of the REEVRP in \Cref{sec:notation}, then present our IP formulation of the problem in \Cref{sec:mip},
and describe its solution via a BPC algorithm in \Cref{sec:bpc}.
For ease of reference for the reader, \Cref{table:notation} provides a centralized summary of the important notation that is formally defined throughout the paper.

\begin{table}[!htb]
    \scriptsize
    \caption{Notation used in the mathematical model.}
    \label{table:notation}
    \begin{tabularx}{\linewidth}{lX}
        \toprule
        \textbf{Set} & \textbf{Description}\\
        \midrule
        $\sG$ & Directed graph representation of the road network \\ %
        $\sV$ & Set of depot and delivery nodes in $\sG$ \\ %
        $\sA$ & Set of directed arcs in $\sG$ \\
        $\sV_C$ & Set of delivery nodes in $\sG$ \\
        $\sK$ & Set of vehicle types \\ %
        $\sK^{+}$ & Set of vehicle types used in set partitioning model \\
        $\sP$ & Set of paths in $\sG$ that start at the origin depot $0$ and end at the destination depot $n+1$ \\
        $\sR$ & Set of all feasible elementary routes in $\sG$ \\
        $\sR^{ng}$ & Set of all feasible $ng$-routes in $\sG$ \\
        $NG_i$, $i \in \sV_C$ & $ng$-set of node $i$ \\
        $\sC_p$, $p \in \sP$ & Set of arcs in the transitive closure of (the subgraph induced by) path $p$ \\
        \midrule
        \textbf{Parameter} & \textbf{Description}\\
        \midrule
        $q_{i}$, $i \in \sV$ & Demand of node $i$ \\
        $s_{i}$, $i \in \sV$ & Service time of node $i$ \\
        $t_{ij}$, $(i, j) \in \sA$ & Travel time along arc $(i, j)$ \\
        $d_{ij}$, $(i, j) \in \sA$ & Travel distance along arc $(i, j)$ \\
        $m_k$, $k \in \sK$ & Number of available vehicles of type $k$ \\
        $Q$ & Demand carrying capacity of each vehicle \\
        $T$ & Maximum working time (hours-of-work) before which vehicles must return to the depot \\
        $D_E$ & EV mode mileage of REEVs (EV range) \\
        $\cev$ & Mileage cost of REEVs in EV mode \\
        $\ccng$ & Mileage cost of REEVs in non-EV mode \\
        $\cconv$ & Mileage cost of CVs \\
        $\alpha_{ip}$, $i \in \sV$, $p \in \sP$ & Total number of times node $i$ is visited on path $p$ \\
        $\beta_{ijp}$, $(i, j) \in \sA$, $p \in \sP$ & Total number of times arc $(i, j)$ appears on path $p$ \\
        $\eta_p$, $p \in \sP$ & Total demand of path $p$ \\
        $\delta_p$, $p \in \sP$ & Total distance of path $p$ \\
        $\tau_p$, $p \in \sP$ & Total processing time (sum of travel and service times) of path $p$ \\
        $c_{pk}$, $p \in \sP$, $k \in \sK$ & Cost incurred when path $p$ is performed by a vehicle of type $k$ \\
        \midrule
        \textbf{Decision Variable} & \textbf{Description}\\
        \midrule
        $\lambda_{rk}$, $r \in \sR^{ng}$, $k \in \sK^{+}$ & Binary variable indicating if route $r$ is performed by a vehicle of type $k$ \\
        \bottomrule
    \end{tabularx}
\end{table}

\subsection{Problem Definition}\label{sec:notation}

The REEVRP is modeled on a directed graph $\sG = (\sV, \sA)$ with nodes $\sV = \{0, 1, \ldots, n, n+1\}$ and arcs $\sA \subseteq \big\{(i, j) \in \sV \times \sV: i \neq j, (i, j) \neq (0, n+1) \big\}$.\footnote{Note that $\sG$ does not necessarily represent the real-world road network. It is generally constructed by transforming the road network graph into a (possibly complete) graph whose nodes are the vertices of the road network corresponding only to the depot and customers.}
Here, nodes $0$ and $n+1$ respectively represent the origin and destination depots of vehicle routes but refer to identical physical locations.
We use $\sV_C = \sV \setminus \{0, n+1\}$ to denote the subset of nodes where deliveries must be made.
Each node $i \in \sV$ is associated with
a demand $q_i > 0$
and a service time $s_i \geq 0$.
Without loss of generality,= we assume that $q_0 = q_{n + 1} = 0$ and $s_0 = s_{n + 1} = 0$.
The travel time along arc $(i, j) \in \sA$ and its distance are denoted by $t_{ij} > 0$ and $d_{ij} > 0$, respectively.
We  assume throughout that the arc lengths satisfy the triangle inequality:
$d_{ij} + d_{jk} \geq d_{ik}$ for all $(i, j), (j, k), (i, k) \in \sA$.

We let $\sK = \{k_h, k_c\}$ denote the set of vehicle types, where $k_h$ and $k_c$ stand for hybrid REEV and CV types, respectively.
There are $m_k \in \mathbb{N}$ vehicles available of type $k \in \sK$, each with a carrying capacity of $Q > 0$ demand units and a maximum allowed working time of $T > 0$ units.
We assume that all REEVs start from the depot with fully charged batteries
and are not allowed to stop en route to recharge them.
The per-mile cost of operating a CV is denoted by $\cconv$, whereas the per-mile costs of operating an REEV in EV and non-EV (i.e., in internal combustion engine or charge-sustaining) modes are denoted by $\cev$ and $\ccng$, respectively.
Throughout, we  assume that $\cev < \ccng \leq \cconv$,
reflecting the relatively higher costs of fossil fuels (gasoline or CNG) compared with battery costs in EV mode (e.g., electricity for charging).
The maximum distance that an REEV can travel in EV mode (also known as its all-electric range) is constrained by $D_E \geq 0$.

We let $\sP$ denote the set of all finite and open paths in $G$ that start at the origin depot $0$, traverse through a subset of $\sV_C$, and end at the destination depot $n+1$.
Specifically, each path $p = (v_0 = 0, v_1, v_2, \ldots, v_l, v_{l+1} = n + 1) \in \sP$ satisfies $l > 0$, $\sA_p \coloneqq \{(v_h, v_{h+1}): h \in \{0, 1, \ldots, l\} \} \subseteq \sA$ and $v_i \in \sV_C$ for $i \in  \{1, \ldots, l\}$.
We use
$\eta_p = \sum_{i=1}^{l} q_{v_i}$ to denote the total demand along $p$,
$\delta_p = \sum_{i=0}^{l} d_{v_i v_{i+1}}$ to denote the travel distance along $p$ and
$\tau_p = \sum_{i=0}^{l} \left(t_{v_i v_{i+1}} + s_{v_i}\right)$ to denote the total duration along $p$ (that is, the total time to process all nodes in $p$).
We use the notation $i \in p$ to indicate that node~$i$ belongs to path~$p$; that is, $i = v_j$ for some $j \in \{0, 1, \ldots, l, l+1\}$. %

If a path $p \in \sP$ is traversed by a vehicle of type $k \in \sK$, then the cost incurred is given as follows:
\begin{equation}\label{eq:cost_definition}
    c_{pk} = \begin{dcases}
        \cev \min\{\delta_p, D_E\} + \ccng \max\{\delta_p - D_E, 0\} & \text{ if } k = k_h \\
        \cconv \delta_p & \text{ if } k = k_c.
    \end{dcases}
\end{equation}
The first term in the first condition on the right-hand side states that REEVs ($k = k_h$) incur a mileage cost of $\cev$ in EV mode but that the total distance  they can travel in this mode is restricted by their EV range, $D_E$.
The second term enforces a mileage cost of $\ccng$ in non-EV mode for any travel distance that exceeds $D_E$.
The second condition on the right-hand side of \cref{eq:cost_definition} states that CVs ($k = k_c$) always incur a per-mile cost of $\cconv$.

A path that is elementary and both time- and capacity-feasible is also called a vehicle route (or simply a route),
and we use $\sR \subseteq \sP$ to denote the set of all such routes.
Specifically, each route $r = (v_0 = 0, v_1, v_2, \ldots, v_l, v_{l+1} = n + 1) \in \sR$ satisfies the following conditions:
\begin{enumerate}[label=(C\arabic*),leftmargin=*]
    \item\label{condition:elementary} $r$ is elementary: $v_i \neq v_j$ for $i, j \in \{1, \ldots, l\}$, $i \neq j$.
    \item\label{condition:capacity} $r$ satisfies the capacity limit: $\eta_r \leq Q$.
    \item\label{condition:duration} $r$ satisfies the duration limit: $\tau_r \leq T$.
\end{enumerate}

A vector of $F$ vehicle routes $R = (r_1, \ldots, r_F)$ in conjunction with a fleet composition vector $\kappa = (\kappa_1, \ldots, \kappa_F)$, where $r_f \in \sR$, $\kappa_f \in \sK$ 
for $f \in \{1, \ldots, F\}$,
is said to define a feasible solution of the REEVRP
if the following are satisfied:
\begin{enumerate}[label=(D\arabic*),leftmargin=*]
    \item\label{condition:partition}
    Routes $r_1, \ldots, r_F$ partition $V_C$; that is, each customer node is visited on exactly one route.
    \item\label{condition:fleet-size}
    The number of routes served by vehicles of type $k$ does not exceed their available number; that is
    $\sum_{f = 1}^F \mathbb{I}[\kappa_f = k] \leq m_k$ for all $k \in \sK$,
    where $\mathbb{I}[\mathcal{E}]$ is the indicator function that evaluates to $1$ if expression $\mathcal{E}$ is true and $0$ otherwise.
\end{enumerate}
The cost of a feasible solution is the sum of the costs of its individual routes, $c(R, \kappa) \coloneqq \sum_{f = 1}^F c_{r_f \kappa_f}$,
and the goal is to determine a feasible solution of minimum cost.

\subsection{Set Partitioning Model}\label{sec:mip}
The central challenge in solving the REEVRP is to model the nonlinear costs in~\cref{eq:cost_definition} incurred by REEVs using a tractable linear integer programming formulation.
For this purpose we first define an extended set of vehicle types $\sK^{+} \coloneqq \{k_e, k_g, k_c\}$.
Here the vehicle type $k_c$ refers to CV types as before.
However, the hybrid REEV type $k_h$ is represented by using two new types that we denote as $k_e$ and $k_g$.
The subtype $k_e$ represents vehicles that operate purely in EV mode at all times; that is, their total individual travel distances do not exceed $D_E$ units.
The subtype $k_g$ represents vehicles that necessarily switch to non-EV mode at some point; that is, their total individual travel distances exceed $D_E$ units.
As we will show, this partition enables a linear cost representation that can be modeled in an integer linear program.

Our set partitioning model uses binary variables $\lambda_{rk}$ that indicate whether a route $r$ is used in the optimal solution and whether it is performed by a vehicle of type $k \in \sK^{+}$.
It is well known (e.g., see \citet{poggi2014}) that one does not sacrifice optimality if these variables are defined over routes $r$ that are not necessarily elementary, that is, those that may not satisfy condition~\ref{condition:elementary}.
In particular, we define these variables over the set $\sR^{ng} \subseteq \sP$ of feasible $ng$-routes (e.g., see \citet{baldacci2011new}), which are defined as follows.
For each $i \in \sV_C$, let $NG_i \subseteq \sV_C$ denote the $ng$-set of node~$i$.
This set may be defined to consist of any subset of $\sV_C$, with the only restriction being that $i \in NG_i$.
It is common to define $NG_i$ to consist of the $\Delta$-nearest neighbors of $i$, with respect to some metric such as the one induced by the distance matrix (usually $\Delta \in \{8, 9, \ldots, 16\}$).
Then, each route $r = (v_0 = 0, v_1, v_2, \ldots, v_l, v_{l+1} = n + 1) \in
\sR^{ng}$ satisfies conditions \ref{condition:ng-elementary},
\ref{condition:capacity} and \ref{condition:duration}, where the elementarity
condition is relaxed as follows:
\begin{enumerate}[label=(C\arabic*$'$),leftmargin=*]
    \item\label{condition:ng-elementary}
    $r$ is $ng$-elementary: if $v_i = v_j$ for some $i, j \in \{1, \ldots, l\}$, $i < j$, then there exists $v_k$ for some $k \in \{i+1, i+2, \ldots, j-1\}$ such that $v_i \notin NG_{v_k}$.
\end{enumerate}

We can now present our IP model.
In the following formulation,
for each path $p = (v_0 = 0, v_1, v_2, \ldots, v_l, v_{l+1} = n+1) \in \sP$,
integer-valued parameters
$\alpha_{ip}$ and $\beta_{ijp}$ count the number of times node~$i$ and arc~$(i, j)$, respectively, appear on a (not necessarily elementary) path $p \in \sP$;
whereas $\sC_p = \{(v_i, v_j) \in \sA: i, j \in \{0, 1, \ldots, l, l+1\}, i < j\}$ denotes the set of arcs in the transitive closure of (the subgraph induced by) $p$.
\begin{subequations}\label{eq:model}
    \begin{align}
        \mathop{\text{minimize}}_{\lambda} \;\; &
        \sum_{r \in \sR^{ng}} \left(
        \cev \delta_r \lambda_{rk_e}
        +\left[\ccng \delta_r - (\ccng - \cev)D_E \right] \lambda_{rk_g}
        +\cconv \delta_r \lambda_{rk_c}
        \right) \\
        \text{subject to} \;
        & \sum_{k \in \sK^{+}}\sum_{r \in \sR^{ng}} \alpha_{ir} \lambda_{rk} = 1, \;\; i \in \sV_C \label[constraint]{eq:degrees} \\
        & \sum_{r \in \sR^{ng}} \left(\lambda_{rk_e} + \lambda_{rk_g}\right) \leq m_{k_h} \label[constraint]{eq:fleet-size-hybrid} \\
        & \sum_{r \in \sR^{ng}} \lambda_{rk_c} \leq m_{k_c} \label[constraint]{eq:fleet-size-conventional} \\
        & \sum_{(i, j) \in \sC_p}\sum_{r \in \sR^{ng}} \beta_{ijr} \lambda_{rk_g} \leq |\sA_p| - 1, \;\; p \in \sR: \delta_p \leq D_E \label[constraint]{eq:ipec} \\
        & \lambda_{rk_e} = 0, \;\; r \in \sR^{ng}: \delta_r > D_E %
        \label[constraint]{eq:ev} \\
        & \lambda_{rk} \in \{0, 1\}, \;\; r \in \sR^{ng}, \; k \in \sK^{+} \label[constraint]{eq:binarity}
    \end{align}
\end{subequations}
In this formulation, \cref{eq:degrees} ensures that each customer is visited exactly once, whereas
\crefrange{eq:fleet-size-hybrid}{eq:fleet-size-conventional} ensure that the fleet size is respected.
Next, \cref{eq:ipec} is an \emph{infeasible path elimination constraint} (IPEC), which enforces that hybrid REEVs of subtype $k_g$ are not assigned to routes whose travel distances are less than or equal to $D_E$ units (a formal proof is provided shortly).
Its counterpart \cref{eq:ev} ensures that hybrid vehicles of subtype $k_e$ are not assigned to routes whose travel distances exceed $D_E$ units.
 \cref{eq:binarity} requires that the variables  be binary.

The following propositions establish the validity of formulation \eqref{eq:model}.
\begin{proposition}
    Every feasible solution of the REEVRP can be mapped to a feasible solution of formulation~\eqref{eq:model} with the same cost.
\end{proposition}
\begin{proof}
    Let $(R, \kappa)$ be a feasible solution of the REEVRP
    where $R = (r_1, \ldots, r_F)$ and $\kappa = (\kappa_1, \ldots, \kappa_F)$ with $r_f \in \sR$ and $\kappa_f \in \sK$ for each $f \in \{1, \ldots, F\}$.
    By definition, $(R,\kappa)$ satisfies conditions \ref{condition:elementary}, \ref{condition:capacity}, \ref{condition:duration}, \ref{condition:partition}, and~\ref{condition:fleet-size}.
    Construct $\lambda$ as follows:
    For each $f \in \{1, \ldots, F\}$, define
    \begin{equation*}
        \lambda_{r_f k_c} = \mathbb{I}[\kappa_f = k_c], \quad
        \lambda_{r_f k_e} = \mathbb{I}[\kappa_f = k_h] \mathbb{I}[\delta_{r_f} \leq D_E], \quad
        \lambda_{r_f k_g} = \mathbb{I}[\kappa_f = k_h] \mathbb{I}[\delta_{r_f} > D_E].
    \end{equation*}
    Also, define $\lambda_{rk} = 0$ for all $r \in \sR^{ng} \setminus \{r_1, \ldots, r_F\}$ and all $k \in \sK^{+}$.
    Now observe that $\lambda$ satisfies \cref{eq:degrees} because $R$ satisfies condition~\ref{condition:partition} and because $\alpha_{ir_f} = \mathbb{I}[i \in r_f]$ (since each $r_f$ is elementary by definition~\ref{condition:elementary});
    it also satisfies \crefrange{eq:fleet-size-hybrid}{eq:fleet-size-conventional} because $R$ satisfies condition~\ref{condition:fleet-size};
    and similarly it satisfies \crefrange{eq:ev}{eq:binarity} by construction of $\lambda$.
    
    To see that \cref{eq:ipec} is also satisfied, fix any route
    $p = (v_0 = 0, v_1, \ldots, v_\ell, v_{\ell + 1} = n + 1) \in \sR$
    such that $\delta_p \leq D_E$, and define the flow variables:
    \begin{equation*}
        x_{ij} = \begin{dcases}
            \sum_{r \in \sR^{ng}} \beta_{ijr} \lambda_{r k_g} & \text{ if } (i, j) \in \sC_p \\
            0 & \text{ otherwise}
        \end{dcases} \;\; \text{ for all } (i, j) \in \sV \times \sV.
    \end{equation*}
    By construction of $\lambda$ and because of the elementarity
    \ref{condition:elementary} and degree \ref{condition:partition} conditions,
    observe that $x_{ij} \in \{0, 1\}$ and $x_{ij} = 1$ if and only if arc $(i,
    j)$ is part of some route $r \in \sR$ such that $\delta_r > D_E$ and it is performed
    by a vehicle of type $k_h$.
    Also, for any $i \in \sV_C$, since the definition of $\alpha$ and $\beta$ implies that
    \[
    \alpha_{ir} = \sum_{j \in \sV_C \cup \{n+1\}: (i, j) \in \sA} \beta_{ijr} = \sum_{j \in \sV_C \cup \{0\}: (j, i) \in \sA} \beta_{jir},
    \]
    the degree condition \ref{condition:partition} requires that $\sum_{j \in \sV} x_{ij} \leq 1$ and $\sum_{j \in \sV} x_{ji} \leq 1$.
    Along with $x_{ij} \in \{0, 1\}$, this implies that in the subgraph $G_x = \{\cup_{i=0}^{\ell + 1} \{v_i\}, \sA_x\}$ induced by $\sA_x \coloneqq \{(i, j) \in \sA : x_{ij} = 1 \}$, each of the nodes $v_1, v_2, \ldots, v_{\ell}$ has out-degree and in-degree equal to $0$ or $1$.
    We now claim that the terminal nodes $u_1, u_2, \ldots, u_m \in p$ of each of the $m \coloneqq \sum_{j=1}^\ell x_{0 v_j}$ %
    subpaths in $G_x$ starting from node $0$ satisfy $u_i \neq n+1$. %
    Indeed, if some $u_i = n+1$, then this implies that there is a path $\pi \in \sP$ for which $\lambda_{\pi k_g} = 1$ (by construction of $\lambda$ and $x$), and hence $\delta_\pi > D_E$.
    But since the travel distances satisfy the triangle inequality and $\pi$ is contained in the transitive closure $\sC_p$ of $p$, we also must have
    $\delta_\pi \leq \delta_p \leq D_E$, resulting in a contradiction.
    Therefore, $u_1, u_2, \ldots, u_m \in \sV_C$, and we have
    \begin{align*}
        \sum_{(i, j) \in \sC_p} x_{ij} = \sum_{j=1}^\ell x_{0 v_j} + \sum_{i=1}^\ell \sum_{j=i+1}^{\ell+1} x_{v_i v_j}
        &= m + \sum_{v \in \{v_1, \ldots, v_\ell\} \setminus \{u_1, \ldots, u_m\}} \sum_{j:(v, j) \in \sC_p} x_{vj} \\
        &\leq m + \sum_{v \in \{v_1, \ldots, v_\ell\} \setminus \{u_1, \ldots, u_m\}} \sum_{j \in V} x_{vj} \\
        &\leq m + (\ell - m)  \;\; \bigg(\text{since } \sum_{j \in \sV} x_{ij} \leq 1 \bigg)\\
        &= |\sA_p| - 1,
    \end{align*}
    showing that \cref{eq:ipec} is also satisfied.
    
    The objective value of the constructed $\lambda$ solution is equal to $\sum_{f = 1}^F \mathrm{cost}_f$, where
    \begin{align*}
        \mathrm{cost}_f &= \begin{dcases}
            \cev \delta_{r_f} \mathbb{I}[\delta_{r_f} \leq D_E]
            +\left[\ccng \delta_{r_f} - (\ccng - \cev)D_E \right] \mathbb{I}[\delta_{r_f} > D_E]
            & \text{ if } \kappa_f = k_h \\
            \cconv \delta_{r_f} & \text{ if } \kappa_f = k_c
        \end{dcases} \\
        &= \begin{dcases}
            \cev \min\{\delta_{r_f}, D_E\} + \ccng \max\{\delta_{r_f} - D_E, 0\}
            & \text{ if } \kappa_f = k_h \\
            \cconv \delta_{r_f} & \text{ if } \kappa_f = k_c
        \end{dcases} \\
        &= c_{r_f \kappa_f},
    \end{align*}
    where the first equality follows by construction of $\lambda$ and the second follows by directly comparing both sides of the equality separately in the cases $\delta_{r_f} \leq D_E$ and $\delta_{r_f} > D_E$.
\end{proof}

\begin{proposition}
    Every feasible solution of formulation~\eqref{eq:model} can be mapped to a feasible solution of the REEVRP with the same cost.
\end{proposition}
\begin{proof}
    Suppose that $\lambda$ is a feasible solution of formulation~\eqref{eq:model}.
    Construct $(R, \kappa)$ as follows:
    \emph{(i)} $F \gets 0$;
    \emph{(ii)} for all $r \in \sR^{ng}$, if $\sum_{k \in \sK^{+}} \lambda_{r k} > 0$, then set $F \gets F + 1$, $r_f \gets r$, and $\kappa_f \gets k_c$ if $\lambda_{r k_c} = 1$ and $\kappa_f \gets k_h$ otherwise.
    We claim that $(R, \kappa)$ is a feasible solution of the REEVRP with the same cost as the objective value of $\lambda$.
    First, observe that all individual routes $r_f$ satisfy conditions \ref{condition:capacity} and \ref{condition:duration} by definition of their domain $\sR^{ng}$.
    Second, the degree \cref{eq:degrees} implies that these routes also satisfy condition \ref{condition:elementary}, and hence the solution $(R, \kappa)$ satisfies condition \ref{condition:partition}.
    Third,
    the fleet size \crefrange{eq:fleet-size-hybrid}{eq:fleet-size-conventional} imply that condition \ref{condition:fleet-size} is also satisfied by $(R, \kappa)$.
    To see that its cost is the same as the objective value of $\lambda$, consider any route $r_f$.
    We consider three possibilities.
    If $\kappa_f = k_c$, then $\lambda_{r_f k_c} = 1$ (by construction), and its objective contribution is $\cconv \delta_{r_f} = c_{r_f \kappa_f}$ (see~\cref{eq:cost_definition}).
    If $\kappa_f = k_h$ and $\delta_{r_f} \leq D_E$, then the IPEC~\eqref{eq:ipec} for $p = r_f$ ensures that $\lambda_{r_f k_g} = 0$ (hence $\lambda_{r_f k_e} = 1$), implying
    that its objective contribution is $\cev \delta_{r_f}$, which is the same as $c_{r_f \kappa_f}$ (see~\cref{eq:cost_definition}).
    If $\kappa_f = k_h$ but $\delta_{r_f} > D_E$, then \cref{eq:ev} ensures that $\lambda_{r_f k_e} = 0$
    (hence $\lambda_{r_f k_g} = 1$) and its corresponding objective contribution is $\ccng \delta_{r_f} - (\ccng - \cev)D_E = \cev \min\{\delta_{r_f}, D_E\} + \ccng \max\{\delta_{r_f} - D_E, 0\} = c_{r_f \kappa_f}$ (see~\cref{eq:cost_definition}).
    This proves the claim.
\end{proof}

\subsection{Branch-Price-and-Cut Algorithm}\label{sec:bpc}

The number of binary variables and constraints in formulation~\eqref{eq:model} scales exponentially with $n$, prohibiting its direct solution by conventional IP solvers.
To address this issue, one can solve formulation~\eqref{eq:model}  using a BPC algorithm, which embeds column generation (also called pricing) and cutting planes in each node of a branch-and-bound search tree.
In the following paragraphs we describe the details of such an algorithm as it applies to formulation~\eqref{eq:model}; see \citet{pecin2017improved,pessoa2020generic} for general references on BPC algorithms for vehicle routing.

\subsubsection{Pricing subproblems}
The linear programming (LP) relaxations of~\eqref{eq:model} at each node of the branch-and-bound tree are solved by column generation.
The latter starts with only a small subset of feasible $ng$-routes for each vehicle type $k \in \sK^{+}$ and then iteratively solves pricing subproblems to add back those routes that have been neglected but that improve the objective value of the current LP relaxation.
For each vehicle type $k \in \sK^{+}$, the pricing subproblem is a shortest-path problem with resource constraints (SPPRCs) where the objective is to identify feasible $ng$-routes with negative reduced costs~\citep{pugliese2013survey}.
Each SPPRC is defined on the original graph $G$ where arc $(i, j) \in \sA$
has a travel cost equal to
$\cev d_{ij}$ for vehicles of type $k_e$,
$\cconv d_{ij}$ for vehicles of type $k_c$,
and $\ccng d_{ij} - \mathbb{I}[i=0](\ccng - \cev)D_E$ for vehicles of type $k_g$,
each of which is appropriately modified to account for contributions coming from the dual values of constraints in the current LP relaxation.
In addition, we define the following resource constraints for the various vehicle types:
\begin{itemize}
    \item Capacity resource constraint for vehicles of type $k \in \sK^{+}$, for which arc $(i, j) \in \sA$ has a resource consumption of $q_j$ units and each node $i \in \sV$ has resource limits of $[0, Q]$
    \item Duration resource constraint for vehicles of type $k \in \sK^{+}$, for which arc $(i, j) \in \sA$ has a resource consumption of $(t_{ij} + s_i)$ units and each node $i \in \sV$ has resource limits of $[0, T]$
    \item Travel distance resource constraint for vehicles of type $k_e$, for which arc $(i, j) \in \sA$ has a resource consumption of $d_{ij}$ units and each node $i \in \sV$ has resource limits of $[0, D_E]$
\end{itemize}
Although the SPPRC is weakly NP-hard, it can be solved efficiently in pseudo-polynomial time by using a labeling-based dynamic programming algorithm~\citep{pugliese2013survey}, the tractability of which depends crucially on the efficiency of its dominance checks.
In this context, all of the above resources are ``disposable,'' which enables the implementation of efficient dominance checks~\citep[Section~5.1.2]{pessoa2020generic}.
Note that a partial path $p$ (``partial'' in the sense that it does not necessarily end at node~$n+1$) dominates another partial path $p'$ whenever for every feasible path extended from $p'$ we can find a corresponding feasible path extension from $p$ with a smaller reduced cost.
More formally, $p$ dominates $p'$ if both end at the same terminal node and satisfy $F_p \subseteq F_{p'}$, $\eta_p \leq \eta_{p'}$ and $\tau_p \leq \tau_{p'}$ (and additionally $\delta_p \leq \delta_{p'}$ for vehicles of type $k_e$), where $F_p$ is the set of so-called forbidden vertices of path $p$~\citep{baldacci2011new}.

In principle, the \emph{lower travel distance limit} of $D_E$ units on vehicles of subtype $k_g$ can also be handled by defining an additional resource whose consumption along arc $(i, j) \in \sA$ is $d_{ij}$ and whose resource limits are equal to $[0, \infty]$ for nodes $i \in \{0\} \cup \sV_C$ and equal to $(D_E,\infty]$ for the destination node~$n+1$.
 Doing so, however, amounts to defining a ``nondisposable'' resource, since an accumulated resource consumption of less than $D_E$ units at node~$n+1$ is not feasible.
In this case the corresponding dominance check is  weak: a partial path~$p$ dominates $p'$ only if $\delta_p = \delta_{p'}$ (along with the aforementioned conditions of ending at the same terminal node, $F_p \subseteq F_{p'}$, $\eta_p \leq \eta_{p'}$ and $\tau_p \leq \tau_{p'}$)~\citep[Section~5.1.2]{pessoa2020generic}.
This reduces the SPPRC to almost complete enumeration of (exponentially many) paths with distances greater than $D_E$ units, resulting in impractical memory requirements.
Therefore, the lower travel distance limit for vehicles of subtype $k_g$ is enforced by using cutting planes via the IPEC defined in~\cref{eq:ipec}, which we describe in \cref{sec:cutting_planes}.

Because dominance checks can become computationally expensive, labels (or partial paths) with similar resource consumption levels are usually saved in a so-called bucket, and dominance is checked only among labels in the same bucket~\citep{pecin2017improved,sadykov2021bucket}.
Other acceleration techniques including heuristic pricing~\citep{fukasawa2006robust}, bidirectional labeling~\citep{righini2008new}, variable fixing~\citep{irnich2010path}, stabilized column generation~\citep{pessoa2018automation}, and dynamic $ng$-sets~\citep{roberti2014dynamic} have also been proposed.

\subsubsection{Cutting planes and separation}\label{sec:cutting_planes}
The solution of a restricted LP at each node of the search tree is iteratively followed by cut separation, where the goal is to add valid cutting planes that are violated at the current LP solution.
These include cuts that are  necessary to ensure solution feasibility, such as the IPEC~\cref{eq:ipec}, as well as those that are only strengthening but not necessary.

Although  the family of the IPEC  has  exponentially many inequalities, we can separate them using the following polynomial-time enumerative path-growing procedure.
Given a (fractional) solution $\lambda^*$ of formulation~\eqref{eq:model},
we first construct the corresponding graph $G_x = (\sV, \sA_x)$,
where
$A_x = \{(i, j) \in \sA : x^*_{ij} > 0\}$,
and the classical flow variables $x^*$ are defined as follows:
\begin{equation*}
    x_{ij}^* = \begin{dcases}
        \sum_{r \in \sR^{ng}} \beta_{ijr} \lambda_{r k_g}^* & \text{ if } (i, j) \in \sA \\
        0 & \text{ otherwise}
    \end{dcases} \;\; \text{ for all } (i, j) \in \sV \times \sV.
\end{equation*}
Starting from node~$0$, elementary paths are grown in a depth-first fashion by moving along the incident outgoing arcs of strictly positive flow.
Each path $p = (v_0 = 0, v_1, v_2, \ldots, v_l, v_{l+1})$ is extended as long as the total flow on its transitive closure, $\sum_{i=0}^l \sum_{j=i+1}^{l+1} x^*_{v_iv_j}$, is strictly greater than $|\sA_p| - 1 = l$ and as long as $v_{l+1} \neq n + 1$; that is, the path has not reached the destination depot.
By construction, every such identified path $p$ satisfying $\delta_p \leq D_E$ defines a corresponding violated IPEC~\cref{eq:ipec}.

Apart from the above necessary IPEC, we also separate the (strengthening but not necessary) rounded capacity inequalities: 
\begin{equation}
    \sum_{(i, j) \in \sA \cap (S \times S)} \sum_{k \in \sK^{+}}\sum_{r \in \sR^{ng}} \beta_{ijr} \lambda_{r k} \leq |S| - \left\lceil\frac{1}{Q}\sum_{i \in S} q_i\right\rceil, \;\;\; S \subseteq \sV_C.
\end{equation}
Furthermore, we  add the following strengthening inequality to the initial master problem:
\begin{equation}\label{eq:str_ineq}
    \sum_{r \in \sR^{ng}} \delta_r \lambda_{r {k_g}} \geq D_E \sum_{r \in \sR^{ng}} \lambda_{r {k_g}}.
\end{equation}
\Cref{eq:str_ineq} states that the total distance covered by all hybrid vehicles of subtype $k_g$ (i.e., the left-hand side)
must be at least as large as $D_E$ times the total number of those vehicles used (i.e., the right-hand side).
We emphasize that adding this inequality along with the infeasible path elimination and rounded capacity constraints does not affect the complexity of the pricing subproblems, since their corresponding dual values affect only the arc costs defining the SPPRC; that is, they are ``robust cuts.''
In contrast, the addition of so-called limited-memory subset row or rank-1 cuts (which are also strengthening but not necessary) significantly increases the complexity of the pricing subproblems; see \citet{jepsen2008subset} and \citet{pecin2017improved,pecin2017new} for details.

\subsubsection{Other implementation details}
Pricing and cut generation are iteratively performed at every node of the search tree until either the resulting (restricted) LP solution is integral and feasible or the addition of cuts does not significantly improve the LP objective value (i.e., tailing off occurs).
At this point, the algorithm proceeds by route enumeration or by branching.
The former refers to enumerating all {elementary routes} with reduced costs less than the current (local) primal-dual optimality gap, whenever this gap is sufficiently small (note that a dual bound can be obtained from any exact solution of the pricing subproblem).
The complexity and implementation of this enumeration step are similar to that of the aforementioned labeling algorithm~\citep{baldacci2008exact}.
If the optimality gap is large, then we proceed to branching, where we prioritize branching on the assignment of customers to vehicle types, given by
$
\sum_{r \in \sR^{ng}} \alpha_{ir} \lambda_{rk}
$.
That is, we branch primarily on the left-hand side of \cref{eq:degrees} for a fixed $k \in \sK^{+}$
and secondarily on the standard flow variable
$
\sum_{k \in \sK^{+}}\sum_{r \in \sR^{ng}} \beta_{ijr} \lambda_{r k}
$.

We utilize the VRPSolver package~\citep{pessoa2020generic} in our implementation with default options except for the addition of at most 10 limited-arc-memory rank-1 cuts (per cut round), %
bidirectional labeling for both heuristic and exact pricing, %
cut tailing-off threshold of $1\%$, and soft and hard time thresholds of $5$ and $10$ seconds, respectively, in the labeling algorithm.

\section{Iterated Tabu Search Metaheuristic}\label{sec:heuristic}
This section describes an iterated tabu search  metaheuristic tailored to solving large-scale REEVRP instances (e.g., with more than a thousand nodes).
This approach is inspired by existing local-search-based metaheuristics for the heterogeneous VRP~\citep{penna2013iterated,subramanyam2020robust}, albeit with some important differences that exploit characteristics of the REEVRP.
A key feature of this ITS approach is its simplicity because it introduces few user-defined parameters and does not require any instance-specific features to accelerate the search process.

ITS can be viewed as an improved version of iterated local search~\citep{lourencco2003iterated}.
The latter refers to the repeated application of local search to a current solution that may be generated either from scratch using a construction heuristic or from perturbing an existing solution.
Local search refers to the repeated use of fundamental moves that transform the current solution into a neighbor solution~\citep{aarts2003local}.
Given a set $M$ of moves, let
$\Omega_M\currentsol$ be the neighborhood of the current solution $\currentsol$ with respect to the move set $M$; that is,
$\Omega_M\currentsol$ is the set of solutions that can be reached from $\currentsol$ by applying a move from $M$.
The crucial components of any local-search-based metaheuristic, therefore, are the definition of $M$ and the exploration of $\Omega_M\currentsol$ using a search algorithm.
For the latter, we use tabu search~\citep{glover1997tabu} since it provides an efficient mechanism to escape local minima in the neighborhood, $\Omega_M\currentsol$.
For the former, we consider $M$ to be the (intra- and interroute) 1-0 relocate, 1-1 exchange, and 2-opt moves that involve the deletion and reinsertion of nodes or edges.
These moves define corresponding node- and edge-exchange neighborhoods that belong to the family of $k$-Opt and $\lambda$-interchange neighborhoods~\citep{funke2005local} with sizes $|\Omega_M(R, k)| = O(n^2)$.

\begin{algorithm}[!htb]
    \small
    \DontPrintSemicolon
    \caption{Iterated tabu search metaheuristic.\label{alg:its}}
    \KwIn{User-defined parameters $\ilsparam{rcl}$, $\ilsparam{tenure}$, $\ilsparam{maxiter}$, $\ilsparam{restart}$, $\ilsparam{perturb}$} 
    \KwOut{Best found solution $\bestsol$}
    $\bestsol \gets (\emptyset, \emptyset)$\;
    \While{\emph{termination criteria not met}}{
        $\currentsol \gets \textsc{Construct Solution}(\ilsparam{rcl})$ \Comment*[r]{Construction phase}
        $\currentsol \gets \textsc{Tabu Search}((R, \kappa), \ilsparam{tenure}, \ilsparam{maxiter})$\;
        $\textrm{counter} \gets 0$, $\tempsol \gets \currentsol$ \Comment*[r]{Perturbation phase}
        \While{\emph{counter} $ < \ilsparam{restart}$}{
            $\currentsol \gets \textsc{Perturb Solution}((R, \kappa), \ilsparam{perturb})$\;
            $\currentsol \gets \textsc{Tabu Search}((R, \kappa), \ilsparam{tenure}, \ilsparam{maxiter})$\;
            \lIf{$\phi\currentsol < \phi\bestsol$}{
                $\bestsol \gets \currentsol$
            }
            \eIf{$\phi\currentsol < \phi\tempsol$}{
                $\mathrm{counter} \gets 0$\;
                $\tempsol \gets \currentsol$\;
            }{
                $\mathrm{counter} \gets \mathrm{counter} + 1$\;
            }
        }
    }
\end{algorithm}

\Cref{alg:its} presents the ITS metaheuristic.
It consists of a construction phase and a perturbation phase.
The former involves the construction of an initial solution that is then improved by using tabu search.
The latter involves iteratively perturbing this solution and again improving it by using tabu search.
If this phase fails to encounter a solution that is better than the best solution found in the current iteration, $\tempsol$, for more than $\ilsparam{restart}$ attempts,
then the construction phase is restarted.
The overall algorithm is terminated when the total run time or the number of iterations exceeds some prescribed limit.

The merit function $\phi\currentsol$ is used to compare different solutions.
When $R = (r_1, \ldots, r_F)$ and $\kappa = (\kappa_1, \ldots, \kappa_F)$,
we define this function as
the weighted sum of its cost and constraint violations:
\begin{equation}\label{eq:penalized_cost}
    \phi\currentsol = c\currentsol + \sum_{f = 1}^F \left(
    \varphi^Q\max\left\{\eta_{r_f} - Q, 0\right\}
    +
    \varphi^T\max\left\{\tau_{r_f} - T, 0\right\}
    \right),
\end{equation}
where $\varphi^Q$ and $\varphi^T$ are penalty parameters that are chosen to be sufficiently large to ensure that a solution $\currentsol$ is better than another one $\tempsol$ (i.e., $\phi\currentsol < \phi\tempsol$) only if it primarily has a lower or equal constraint violation and secondarily a lower cost.
This approach allows the algorithm to handle infeasible solutions whose routes may not satisfy the capacity~\ref{condition:capacity} or duration~\ref{condition:duration} conditions.
Moreover, large values of $\varphi^Q$ and $\varphi^T$ discourage the algorithm from re-entering the infeasible region as soon as a feasible solution is encountered.
The parameters $\ilsparam{rcl}$, $\ilsparam{tenure}, \ilsparam{maxiter}$, and $\ilsparam{perturb}$ are used as inputs to the construction heuristic, tabu search, and perturbation mechanisms, respectively, which we describe next.

\subsection{Construction Heuristic}\label{sec:construction_heuristic}
The \textsc{construct solution} procedure constructs an REEVRP solution that always satisfies the set partitioning~\ref{condition:partition} and fleet availability~\ref{condition:fleet-size} constraints; however, the individual routes need not necessarily satisfy the vehicle capacity~\ref{condition:capacity} or duration~\ref{condition:duration} limits.
The procedure adds customers into an empty solution using the following sequential insertion heuristic.

Each iteration of the insertion heuristic consists of two steps.
First, if the number of REEV routes in the current (partial) solution is less than its available number $m_{k_h}$, then an empty REEV route is constructed; otherwise, an empty CV route is constructed provided again that the number of CV routes is less than $m_{k_c}$.
Second, unrouted customers are repeatedly inserted into this (current) route until either all customers have been routed or there is no feasible insertion position in the route; in other  words, inserting any unrouted customer at any position in the route would  exceed either the vehicle capacity $Q$ or the duration limit $T$.
Specifically, each unrouted customer that can be feasibly inserted into at least one position is first added to a so-called restricted candidate list.
A random customer is then selected from this list and inserted in the position that minimizes a greedy function, defined as a randomly weighted sum of the corresponding {insertion cost}, the residual capacity $Q - \eta_r$, and the residual duration $T - \tau_r$.
The restricted candidate list is cardinality-based; that is, it is sorted according to the greedy function, and its length is not allowed to exceed $\ilsparam{rcl}$.
This parameter controls the extent of randomization and greediness with our implementation using $\ilsparam{rcl} = 3$.
If no empty route can be constructed in the first step, then any remaining unrouted customer is inserted into an existing route (and corresponding position) for which a randomly weighted sum of the insertion cost, capacity violation, and duration violation is minimized. 

\subsection{Tabu Search}
The \textsc{tabu search} procedure enhances the performance of classical local search by \textit{(i)} allowing non-improving moves and \textit{(ii)} potentially disallowing improving moves.
This enhancement, which is achieved by  using a short-term memory that is also known as a {tabu list}, provides a formal mechanism for the current solution to potentially escape from local minima in the neighborhoods defined with respect to the considered elementary moves.
The tabu list keeps track of the most recently visited solutions in the search history and prevents revisiting them for a predefined number of local search iterations, $\ilsparam{tenure}$.
Any potential solution that has been visited in the last $\ilsparam{tenure}$ iterations is marked ``tabu'' (or forbidden) and inserted into the tabu list, to prevent the algorithm from cycling and repeatedly visiting the same solutions.
In fact, a neighbor solution that is in the tabu list can be visited only if certain \emph{aspiration criteria} are met; in particular, the tabu status of a solution is overridden only if it improves upon the best-encountered solution.
The tabu search terminates if it performs $\ilsparam{maxiter}$ local search iterations without observing any further improvement.
Our implementation uses $\ilsparam{tenure} = 20$ and $\ilsparam{maxiter} = 100$.

At each iteration,
\emph{(i)} we randomly select a neighborhood (intra- or interroute relocate, exchange, or 2-opt);
\emph{(ii)} traverse it in lexicographic order applying pruning mechanisms based on both feasibility and gain; and
\emph{(iii)} use a first-improvement strategy to replace the current solution.
The per-iteration run time is $\mathcal{O}(n^2)$ in the worst case, since the size of each of the aforementioned neighborhoods is $\mathcal{O}(n^2)$ and each move is evaluated in constant time by simply updating the current value of $\phi\currentsol$.
Specifically, in addition to storing the values of $\delta_{r_f}$, $\eta_{r_f}$ and $\tau_{r_f}$ for the routes $r_f$, $f \in \{1, \ldots, F\}$ of the current solution, we  store and incrementally update (in constant time) the value of $\max\{\delta_{r_f} - D_E, 0\}$ for REEV routes $r_f$ with $\kappa_f = k_h$, which enables the efficient evaluation of the travel cost for REEV routes as well (see~\cref{eq:cost_definition}).

\subsection{Perturbation Mechanism}

The goal of the \textsc{perturb solution} procedure is to apply ``large'' perturbations to the current solution $\currentsol$ so as to encourage a low likelihood of encountering the resulting solution by applications of tabu search alone.
To that end, the procedure first removes the route $r$ of vehicle type $k$ from the current solution that has the maximum {average cost per unit of carried load}, defined as $c_{rk}/\eta_r$.
In addition, it  removes any routes $r_1', r_2', \ldots$ that are ``sufficiently close'' to $r$, namely, routes for which $\Delta(r, r_i') \coloneqq \max_{(i, j) \in r \times r_i'} d_{ij} < \ilsparam{perturb}$.
If all routes $r' \neq r$ satisfy $\Delta(r, r') \geq \ilsparam{perturb}$, then a route $r'' \in \argmin_{r' \neq r} \Delta(r, r')$ is removed from the current solution.
All customers that are visited on the deleted routes are then considered to be unrouted and are added back to the current partial solution by using the same sequential insertion heuristic described in \Cref{sec:construction_heuristic}.
The parameter $\ilsparam{perturb}$ determines the extent of perturbation, with larger values corresponding to higher extents of perturbation.
Our implementation uses $\ilsparam{perturb} = 0.6 \max_{(i, j) \in \sV_C \times \sV_C} d_{ij}$.

\section{Case Studies}\label{sec:experiments}

The goals of this section are threefold.
First, \Cref{sec:design} describes the details of our experimental design including data generation.
Second, \Cref{sec:performance} illustrates the numerical performance of the proposed BPC and ITS algorithms.
Third, \Crefrange{sec:deployment}{sec:service_time} provide key managerial and policy insights.

\subsection{Design of experiments} \label{sec:design}

As mentioned in \Cref{sec:description}, we assume that Amazon, FedEx, UPS, and USPS are tasked with delivering packages on a typical weekday in the Chicago metropolitan area. The demand and road network data are obtained from POLARIS~\citep{auld2016polaris}. In the POLARIS Chicago demand database, 4 million households  correspond to 10 million individuals. According to the National Household Travel Survey data from 2017~\citep{nhts}, each household generates roughly one order per week. We therefore assume that on the considered day, roughly 600,000 households (which we also refer to as customers) require an e-commerce delivery service.
We randomly draw their locations from the database in a uniform manner.

In addition, we identify 53 depot locations of the aforementioned LSPs from publicly available sources. These sources indicate that Amazon, FedEx, UPS, and USPS operate 17, 10, 13, and 13 depots in the area, respectively. We assume customers are randomly distributed to Amazon, FedEx, UPS, and USPS following 21, 16, 24, and 39 percentage shares, respectively~\citep{pitney}.
To efficiently carry out these delivery tasks, we assume that each provider first solves an assignment problem to determine the optimal depot-to-customer allocation. %
The objective of this assignment problem is to minimize the sum of out-and-back travel times between customers and depots, subject to aggregate depot-level capacity constraints while ensuring that each customer is assigned to some depot. 

Each of the 53 depots defines an REEVRP instance averaging 11,447 customers. \Cref{fig:REEVRP_instance} visualizes an example instance over the Chicago map. 
These instances are so large that simply constructing and loading each travel time matrix in memory---let alone optimizing over them---requires more than 130~million shortest-path calculations.
This necessitates some sort of network aggregation.
Moreover, and perhaps more important, we do not have access to accurate network information at a fine microscopic level.
Indeed, since POLARIS is a mesoscopic traffic simulation software, it does not account for local street-level roads where customers are typically located; instead it considers  only  interstates, principal and other arterials, and major collectors.
Therefore, we downscale the original network by aggregating customers located along individual network arcs by introducing nodes known as \emph{superlocations} at the centers of these arcs. \Cref{fig:super_location} details the placement of superlocations on a small portion of the network. 
After we perform this aggregation, the resulting depot-level instances vary in size from 92 to 3,706 superlocations, with an average and standard deviation of 1,345 and 872, respectively.
The travel time matrix between superlocations is computed by  using Dijkstra's shortest paths algorithm, using POLARIS-provided network information (i.e., nodes, arcs, arc-speeds, and arc-lengths).

\begin{figure}[!htb]
    \begin{center}
        \includegraphics[width=0.4\linewidth]{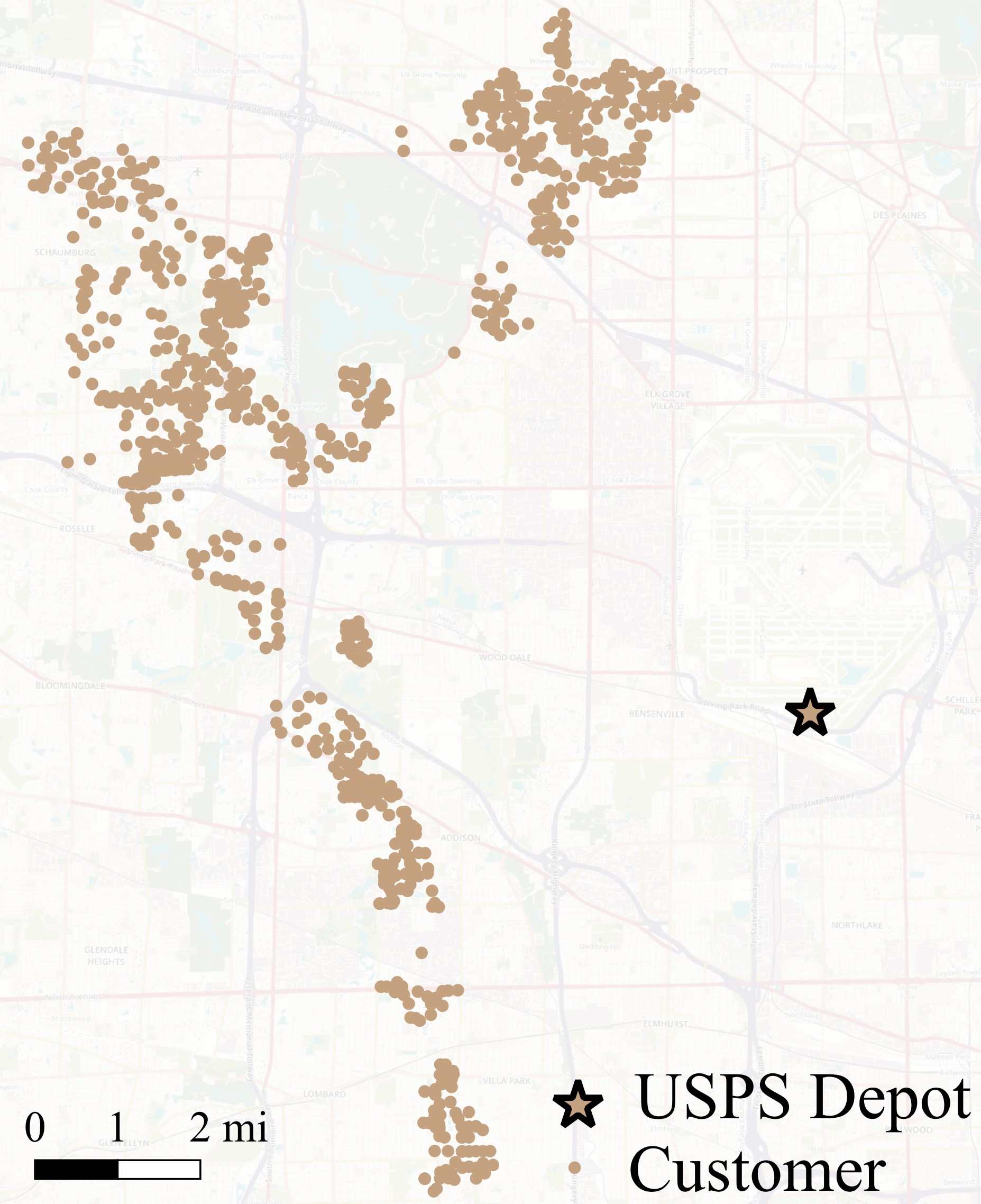}
    \end{center}
    \caption{Layout of an illustrative depot-level REEVRP instance.\label{fig:REEVRP_instance}}
\end{figure}

\begin{figure}[!htb]
    \begin{center}
        \includegraphics[width=0.6\linewidth]{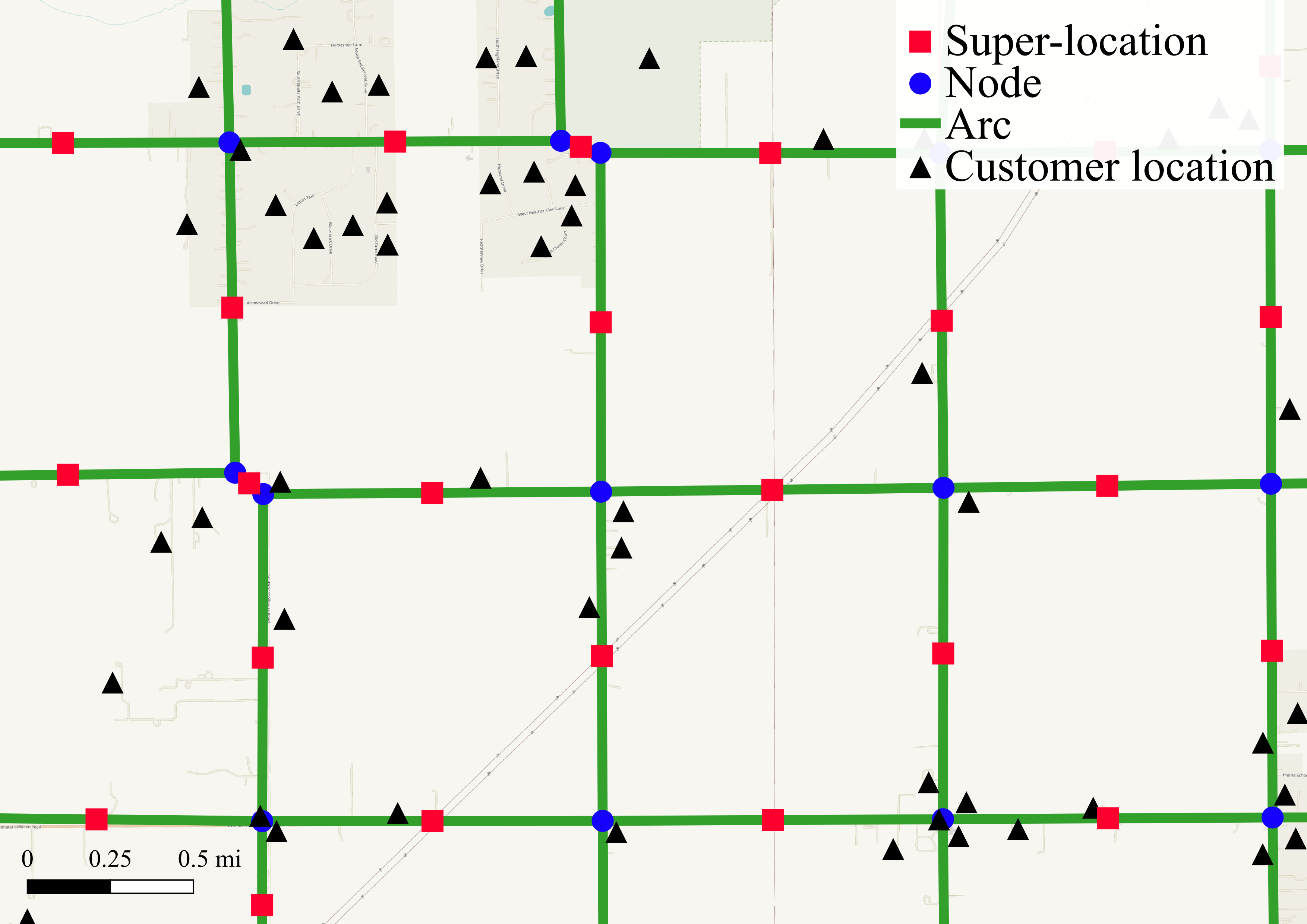}
    \end{center}
    \caption{Illustration of superlocations and their relation to other network components.\label{fig:super_location}}
\end{figure}

The total number of packages to be delivered in the $i^\text{th}$ superlocation (summed over all customers located there) defines its corresponding demand $q_i$ (since every customer is associated with only one package).
The customer-to-customer travel times for customers that are located inside the same superlocation are estimated by dividing their Manhattan distance with a preset low speed of 15~mph because trucks move significantly slower on these minor roads.
The optimal visiting sequence of customers in each superlocation is found by solving a {small} traveling salesman problem.
In particular, the sum of the corresponding total travel time and a per-customer package dropping time of $\sigma$ minutes defines the service time $s_i$ for the $i^\text{th}$ superlocation.
The sum of the corresponding total travel distance is incorporated into the distance matrix entry, $d_{ji}$, for all other superlocations $j\neq i$.

For each depot, we initially consider only two vehicle types: REEVs and CVs.
Subsequently, we also investigate the effects of replacing all REEV trucks with pure BEV trucks, to study the trade-off between these two powertrains.
Note that pure BEV trucks can be easily modeled as a special case of REEV trucks by disallowing the use of $k_g$ subtypes in the BPC algorithm and by setting $\ccng = +\infty$ (or a sufficiently large number) in the ITS algorithm.
Doing so prohibits the hybrid vehicles from using their range extenders.
In our experiments we also disallow BEV battery recharging along their routes but increase their EV range to enable a fairer comparison with their REEV counterparts.
Nonetheless, since pure BEVs have higher ownership- and infrastructure-related fixed costs that are not accounted for in this study, their estimated costs may be slightly undercounted.

The mileage cost of diesel-powered CVs is estimated to be
$\cconv = c_D / \mu_D$, where $c_D$ is the cost of diesel per gallon (USD/gallon) and $\mu_D$ is the fuel economy in miles per gallon (MPG).
The mileage cost of REEVs in EV mode is estimated to be
$\cev = c_E \cdot \mu_E$, where $c_E$ is the cost of electricity (USD/kWh) and $\mu_E$ is the vehicles' electricity consumption rate in EV mode (kWh/mile).
Their mileage cost in CNG mode is estimated to be
$\ccng = c_G / \mu_G$, where $c_G$ is the cost of CNG per gasoline gallon equivalent (USD/GGE) and $\mu_G$ is its charge-sustaining mode fuel economy in miles per gasoline gallon equivalent (MPGe).
We use the CNG range extender medium-duty vehicle developed by \citet{cerc} as the reference REEV type, for which
$\mu_E = 1.14$ kWh/mile,
$\mu_G = 10.1$ MPGe,
and $D_E = 33$ miles (in the REEV/CV case) and $150$ miles (in the BEV/CV case).
Based on year~2020 data from \citet{eia} and July~2021 data from \citet{afdc}, we set
$c_E = 9.90$ \textcent/kWh,
$c_G = 2.22$ USD/GGE,
and
$c_D = 3.36$ USD/gallon.
The average fuel economy of diesel-powered parcel delivery trucks is reported between 7.9 and 10.2 MPG~\citep{lammert2009twelve,barnitt2011fedex}; therefore, we set
$\eta_D = 9.0$ MPG.
The baseline deployment scenario sets
$\sigma = 4$ minutes,
$Q = 120$ packages,
$T = 10$ hours,
and $m_{k_h} = 25$ hybrid vehicles per depot (in both REEV/CV and BEV/CV case studies).
We investigate regionwide operational impacts and sensitivity of our findings to variations in several of these parameter values%
in \Crefrange{sec:deployment}{sec:service_time}.
In each case study reported, we run the ITS metaheuristic 10~times (each with a time limit of one~hour) and use the best of these 10~runs to estimate the corresponding optimal solution.

\subsection{Computational Performance of the BPC and ITS Methods}\label{sec:performance}

We begin by analyzing the computational performance of the proposed BPC and ITS algorithms as a function of several operational parameters.
To that end, we design a testbed of smaller REEVRP instances as follows.
First, for each depot in a random subset of 25 depots from \Cref{sec:design}, we define two smaller subnetworks by randomly drawing $n \in \{50, 100\}$ superlocations from these instances.
We then consider six scenarios, parameterized as per \Cref{table:performance_set}, for each of the $25\times 2$ subnetworks, resulting in a total of $25 \times 2 \times 6 = 300$ test instances.

\begin{table}[!htbp]
    \footnotesize
    \centering
    \caption{Parameterization of operational inputs in scenarios defining the smaller-sized test instances.}
    \label{table:performance_set}
    \begin{tabularx}{\textwidth}{lCCCC}
        \toprule
        Scenario & $m_{k_h}$ (\# per depot)  & $Q$ (packages) & $D_E$ (miles) & $T$ (hours) \\
        \midrule
        1 (baseline)          & 5   & 80 & 33 & 8\\
        2 (all REEV)          & $n$ & 80 & 33 & 8\\
        3 (all CV)            & 0   & 80 & 33 & 8\\
        4 (higher capacity)   & 5   & 120 & 33 & 8\\
        5 (higher EV range)   & 5   & 80 & 66 & 8\\
        6 (higher work hours) & 5   & 80 & 33 & 10\\
        \bottomrule
    \end{tabularx}
\end{table}

All runs were conducted on Intel~Xeon servers.
The ITS algorithm was implemented in C++, and each ITS run was performed on a 2.3~GHz server, limited to 10~minutes of run time on a single CPU thread.
Each BPC run was performed on a 3.1~GHz server, limited to a single CPU thread with 10~GB RAM and 1~hour of computational time.
We used the Julia interface of VRPSolver~\citep{pessoa2020generic} as the BPC backend (e.g., for column generation and branching) along with a Julia routine to add cuts.
All subordinate optimization problems were solved by using the CPLEX~12.10.0 optimizer \citep{CPLEX}.
Each BPC run was initialized with an upper bound corresponding to the objective value of a solution that was computed by  using a single run of the ITS algorithm.

\Cref{table:bpc_summary} reports the numerical performance of the BPC algorithm across all 300 test instances.
For each scenario, the ``Optimal'' column reports the number of instances that were solved to optimality along with the average solution time in seconds. 
The ``Nonzero optimality gap'' column reports the number of instances that could not be solved to optimality in the time limit of one hour, along with the corresponding average optimality gap, defined as $(z_\text{ub}-z_\text{lb})/z_\text{ub} \times 100\%$, where $z_\text{ub}$ and $z_\text{lb}$ are the global upper and lower bounds at termination, respectively.
These quantities are also graphed in the box plots in \Cref{fig:bpc}.
The ``\# IPEC'' column reports the average number of violated IPEC \cref{eq:ipec} that were added in the algorithm.
 We observe that the overall performance of the BPC algorithm when compared with the baseline scenario is fairly consistent across scenarios but deteriorates as the EV range $D_E$ or the number of nodes $n$ increases.
Nevertheless, we observe that the BPC algorithm can find and certify optimal solutions in 113 out of 300 instances in less than 10~minutes (on average)
and can compute reasonably tight optimality gaps of roughly 5\% in other cases.

\begin{table}[!thbp]
    \footnotesize
    \centering
    \caption{Summary of computational performance of the BPC algorithm.}
    \label{table:bpc_summary}
    \begin{tabularx}{0.96\textwidth}{lC*{5}{R}}
        \toprule
        Scenario & \# Total Inst. &  \multicolumn{2}{c}{Optimal} &  \multicolumn{2}{c}{Nonzero gap} &          \# IPEC \\
        \cmidrule(r){3-4}\cmidrule(l){5-6}
        &                &     \# Inst. &    Time (sec) &       \# Inst. &        Gap (\%) & \\
        \midrule
        1 &             50 &           24 &         220.6 &             26 &            3.20 &              5.9 \\
        2 &             50 &            5 &         572.7 &             45 &           10.42 &           2,175.7 \\
        3 &             50 &           30 &         687.1 &             20 &            1.57 &              0.0 \\
        4 &             50 &           22 &         499.6 &             28 &            2.95 &             26.1 \\
        5 &             50 &           10 &         611.9 &             40 &            5.15 &             71.6 \\
        6 &             50 &           22 &         218.2 &             28 &            2.73 &              7.7 \\
        $n=50$ &            150 &          102 &         320.0 &             48 &            7.47 &            602.0 \\
        $n=100$ &            150 &           11 &        1,640.0 &            139 &            4.24 &            160.4 \\
        \midrule
        All &            300 &          113 &         448.5 &            187 &            5.07 &            381.2 \\
        \bottomrule
    \end{tabularx}
\end{table}

\begin{figure*}[!htb]
    \centering
    \subfloat[Optimality gaps\label{fig:bpc_gaps}]{%
        \includegraphics*[width=0.45\textwidth,height=\textheight,keepaspectratio]{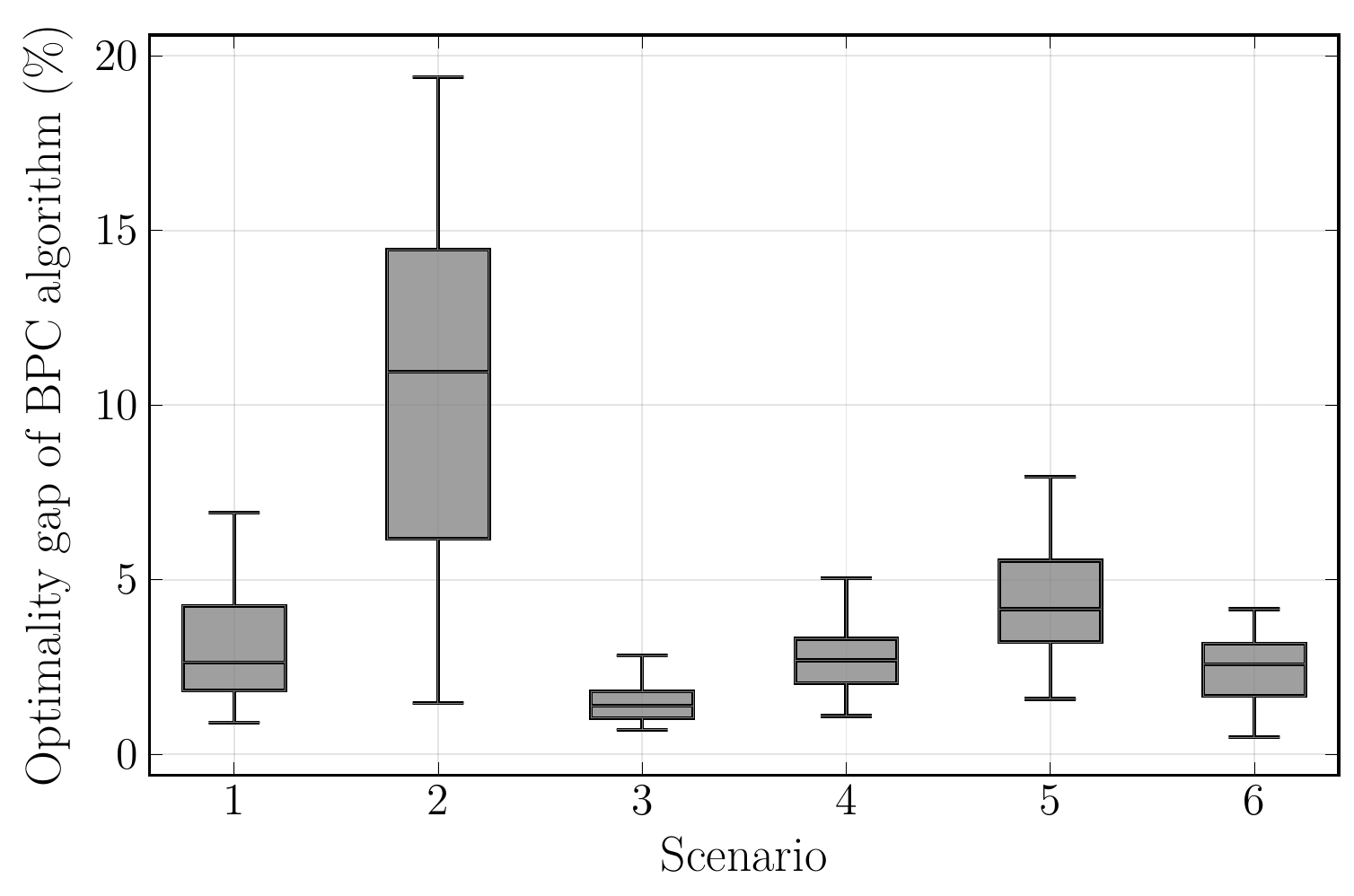}}\hfil
    \subfloat[Solution times\label{fig:bpc_times}]{%
        \includegraphics*[width=0.45\textwidth,height=\textheight,keepaspectratio]{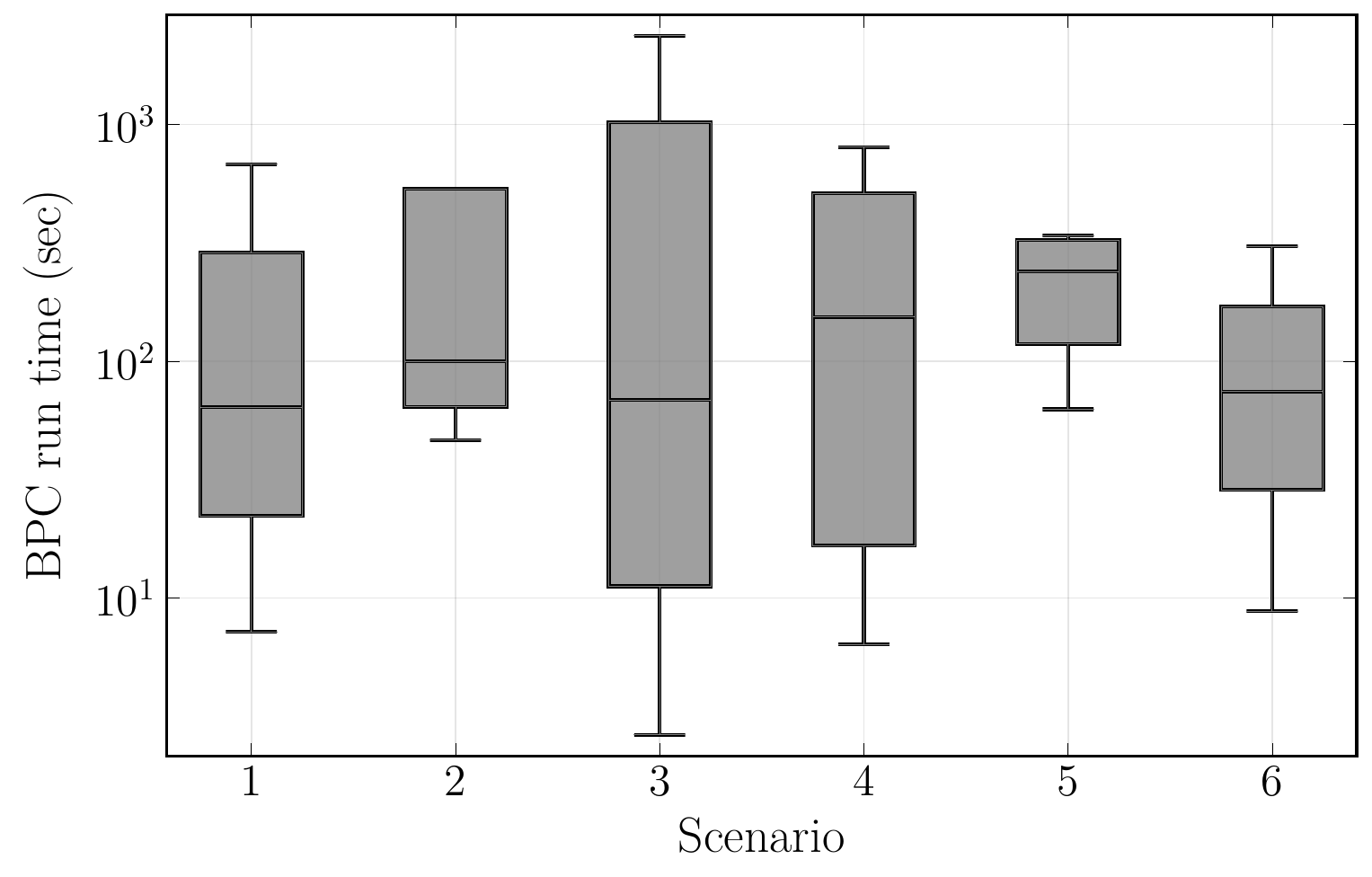}}
    \caption{Guaranteed optimality gaps (for instances not solved to optimality in one~hour) and solution times (for optimally solved instances) of the BPC algorithm across the 6 scenarios in \Cref{table:performance_set}.} \label{fig:bpc}
\end{figure*}

\Cref{table:its_summary} reports the computational performance of the ITS algorithm across all 300 test instances, with each run repeated 10 times using different random seeds.
For each scenario, the ``\# Optimal'' column reports the number of instances for which the optimal objective value is known (from \Cref{table:bpc_summary}) and for which the ITS algorithm found a solution matching the optimal objective value.
The ``Best 10'' and ``Avg. 10'' columns report, respectively, the best and average optimality gaps of the ITS solutions with respect to the global lower bounds from \Cref{table:bpc_summary}.
These quantities are defined as
$(z_\text{best}-z_\text{lb})/z_\text{best} \times 100\%$
and
$(z_\text{avg}-z_\text{lb})/z_\text{avg} \times 100\%$,
where $z_\text{best} = \min\{z_1, \ldots, z_{10}\}$
and
$z_\text{avg} = (z_1+\ldots+z_{10})/10$ 
are the best and average of the objective values,
$z_1, \ldots, z_{10}$,
from the 10 ITS runs.
The ``Avg. dev. (\%)'' and ``Avg. time (sec)'' columns report the average deviation, $\frac{1}{10}\sum_{k=1}^{10} (z_k-z_\text{best})/z_\text{best} \times 100\%$
and average time to find the best solution, respectively.
These quantities are also graphed in the box plots in \Cref{fig:its}.

\begin{table}[!thb]
    \footnotesize
    \centering
    \caption{Summary of computational performance of the ITS algorithm.}
    \label{table:its_summary}
    \begin{tabularx}{\textwidth}{lCRCCCR}
        \toprule
        Scenario & \# Total Inst. & \# Optimal &              \multicolumn{2}{c}{Optimality gap} &            Avg. dev. &            Avg. time \\
        \cmidrule{4-5}
        &                &                        &                Best 10 &  Avg. 10 &                 (\%) &                (sec) \\
        \midrule
        1 &             50 &  19/24 &                   1.43 &                   1.73 &                 0.32 &                199.0 \\
        2 &             50 &   5/5 &                   9.28 &                   9.39 &                 0.12 &                165.1 \\
        3 &             50 &  21/30 &                   0.61 &                   0.73 &                 0.11 &                179.6 \\
        4 &             50 &  19/22 &                   1.47 &                   1.73 &                 0.26 &                191.5 \\
        5 &             50 &   8/10 &                   3.87 &                   4.15 &                 0.29 &                226.4 \\
        6 &             50 &  19/22 &                   1.30 &                   1.56 &                 0.26 &                187.6 \\
        $n=50$ &            150 &  90/102 &                   2.41 &                   2.45 &                 0.04 &                 91.8 \\
        $n=100$ &            150 &   1/11 &                   3.58 &                   3.98 &                 0.41 &                291.3 \\
        \midrule
        All &            300 &  91/113 &                   2.99 &                   3.22 &                 0.23 &                191.5 \\
        \bottomrule
        
    \end{tabularx}
\end{table}

\begin{figure*}[!htb]
    \centering
    \subfloat[Average deviation\label{fig:its_deviation}]{%
        \includegraphics*[width=0.45\textwidth,height=\textheight,keepaspectratio]{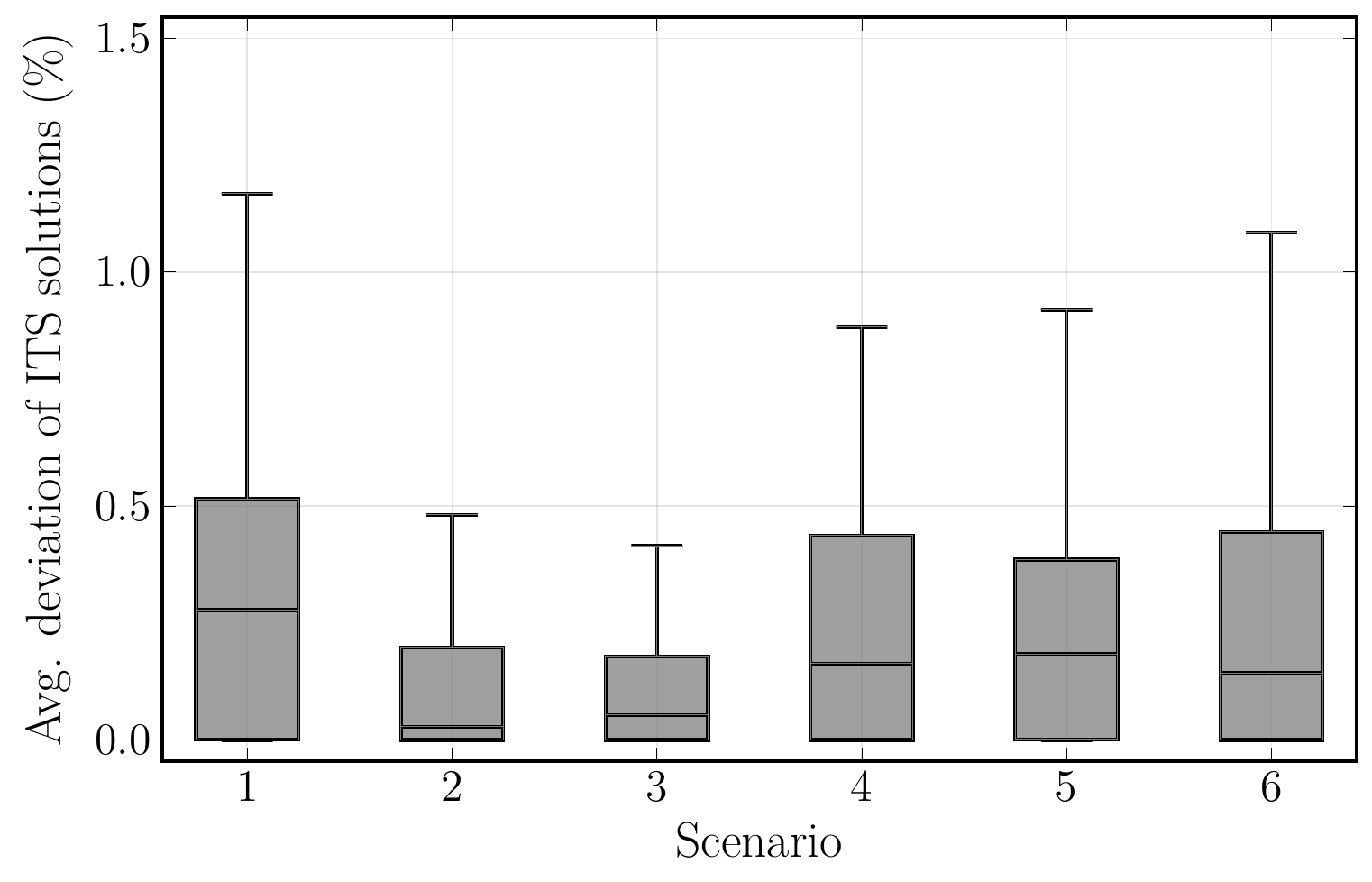}}\hfil
    \subfloat[Solution times\label{fig:its_timetobest}]{%
        \includegraphics*[width=0.45\textwidth,height=\textheight,keepaspectratio]{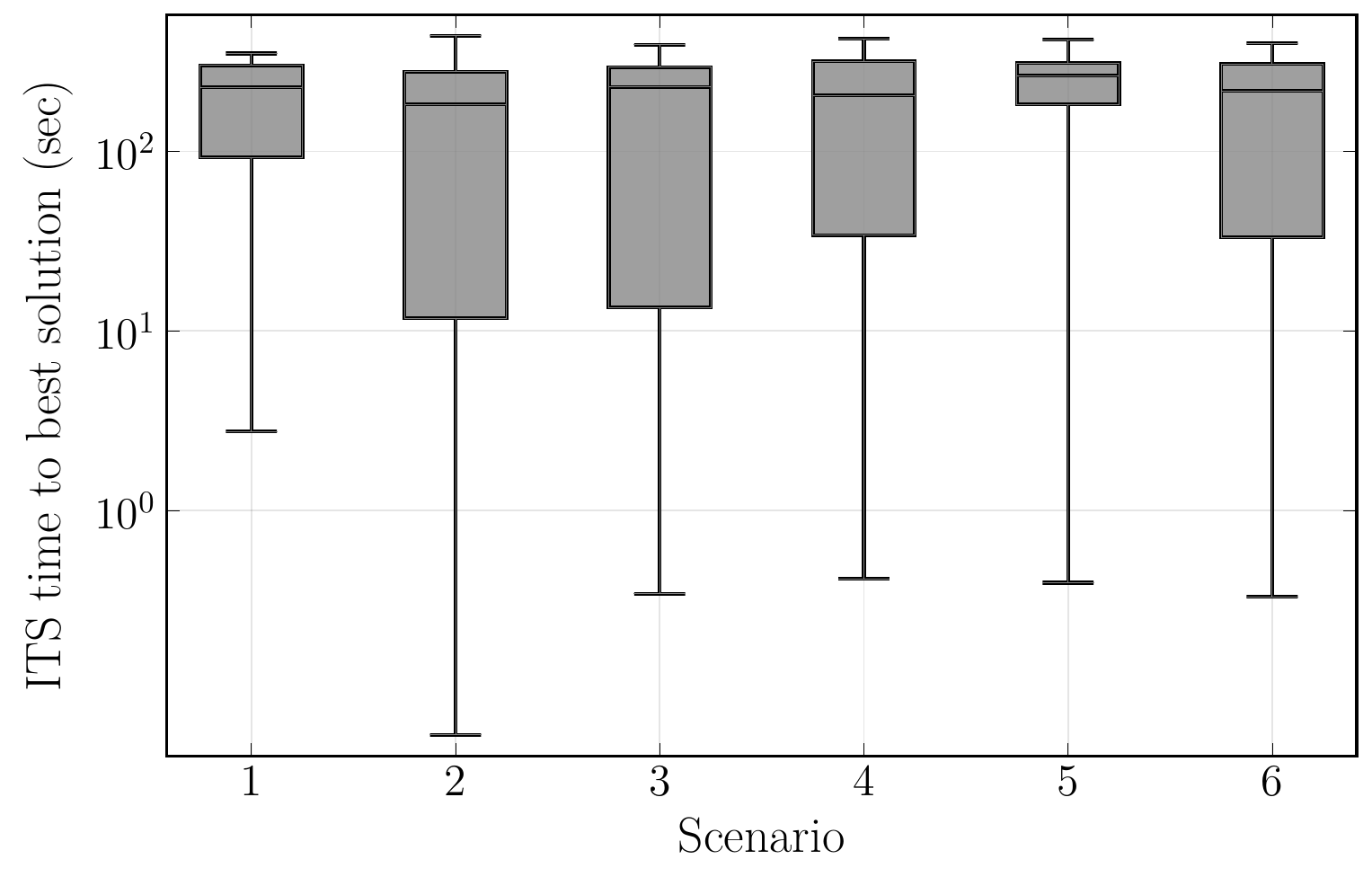}}
    \caption{Average deviation and run times of the ITS metaheuristic across the 6 scenarios in \Cref{table:performance_set}.} \label{fig:its}
\end{figure*}

We make the following observations from \Cref{table:its_summary}.
First, the ITS algorithm can find the optimal solution in 91 out of 113 instances for which the optimal value is known.
This indicates that the reported optimality gaps in \Cref{table:bpc_summary} and \Cref{table:its_summary} are potentially quite conservative and limited by the lower bounds from the BPC algorithm rather than the (suboptimality of the solutions computed by the ITS algorithm.
Second, the low deviation of roughly $0.23\%$ across the 10 runs indicates that the ITS method is fairly robust to the initial random seed. 
Third, the run times are reasonably small across the different scenarios, with the biggest distinction being the problem size, where the run time can increase from roughly 90~seconds for $n=50$ nodes to 3~minutes for $n=100$ nodes.

\subsection{Impact of EV Deployment on Systemwide Performance Metrics}\label{sec:deployment}

We now analyze different deployment scenarios to assess the impact of REEVs and BEVs on the operational energy cost, VMT, and VHT.
The first case considers all regional-level deliveries (i.e., the 53-depot Chicago metropolitan area) using REEVs and CVs.
The second case considers the impact of replacing all REEVs with BEVs.
In the following sections we refer to these two cases simply as ``REEV'' and ``BEV,'' respectively.
In each of the two cases, we consider six scenarios of $0$, $25$, $50$, $100$, $200$, and $2,000$ EVs per depot (be it REEV or BEV).
Note that the number $2,000$ here is meant to represent an arbitrarily large value;  no depot uses so many vehicles.
We refer to the scenarios of $0$ EVs and $2,000$ EVs as ``All CV'' and ``All EV,'' respectively.

\begin{figure*}[!htb]
    \centering
    \subfloat[Cost\label{fig:Costgivendeployment}]{%
        \includegraphics*[width=0.32\textwidth,height=\textheight,keepaspectratio]{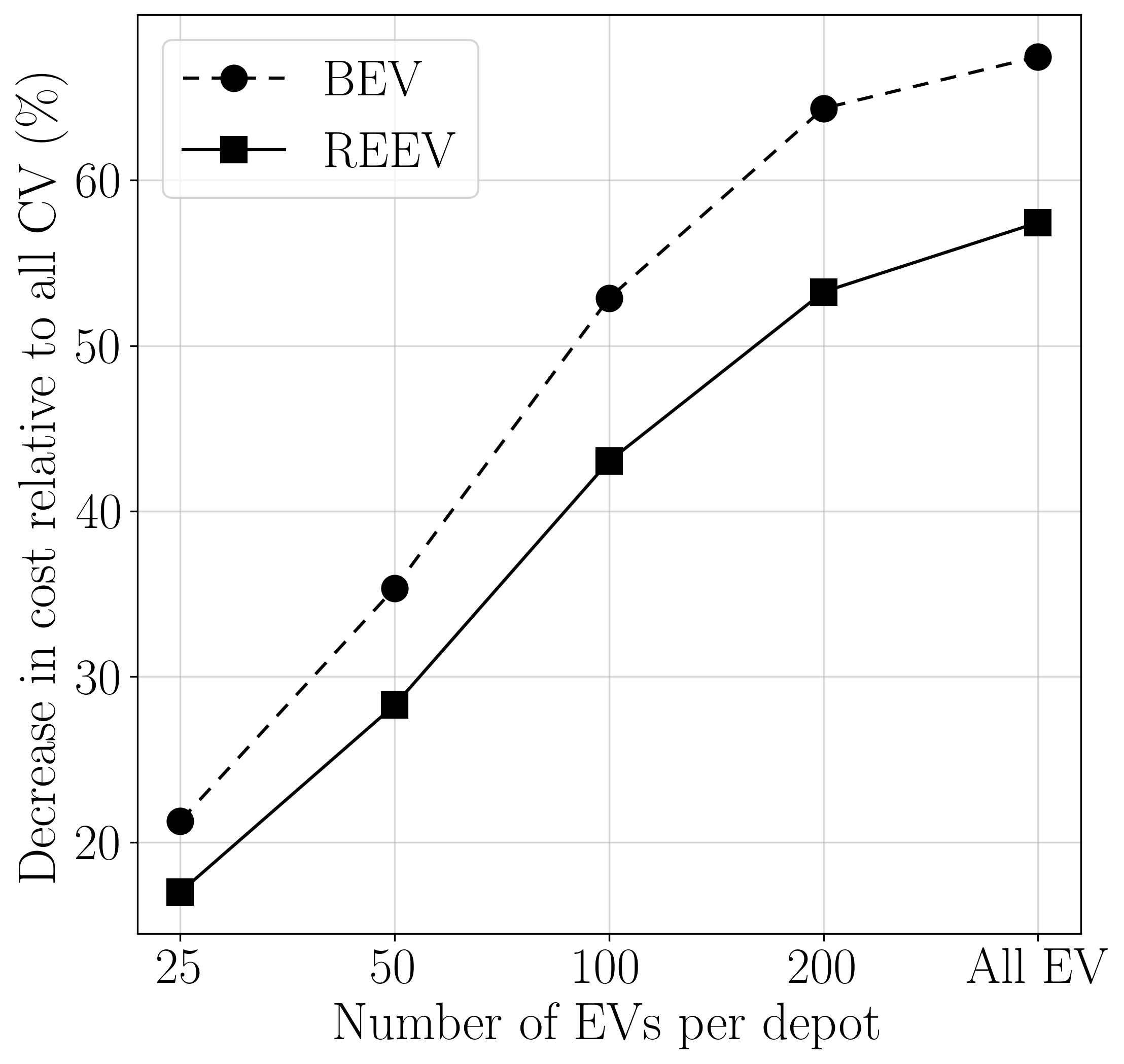}}
    \subfloat[VMT\label{fig:VMTgivendeployment}]{%
        \includegraphics*[width=0.33\textwidth,height=\textheight,keepaspectratio]{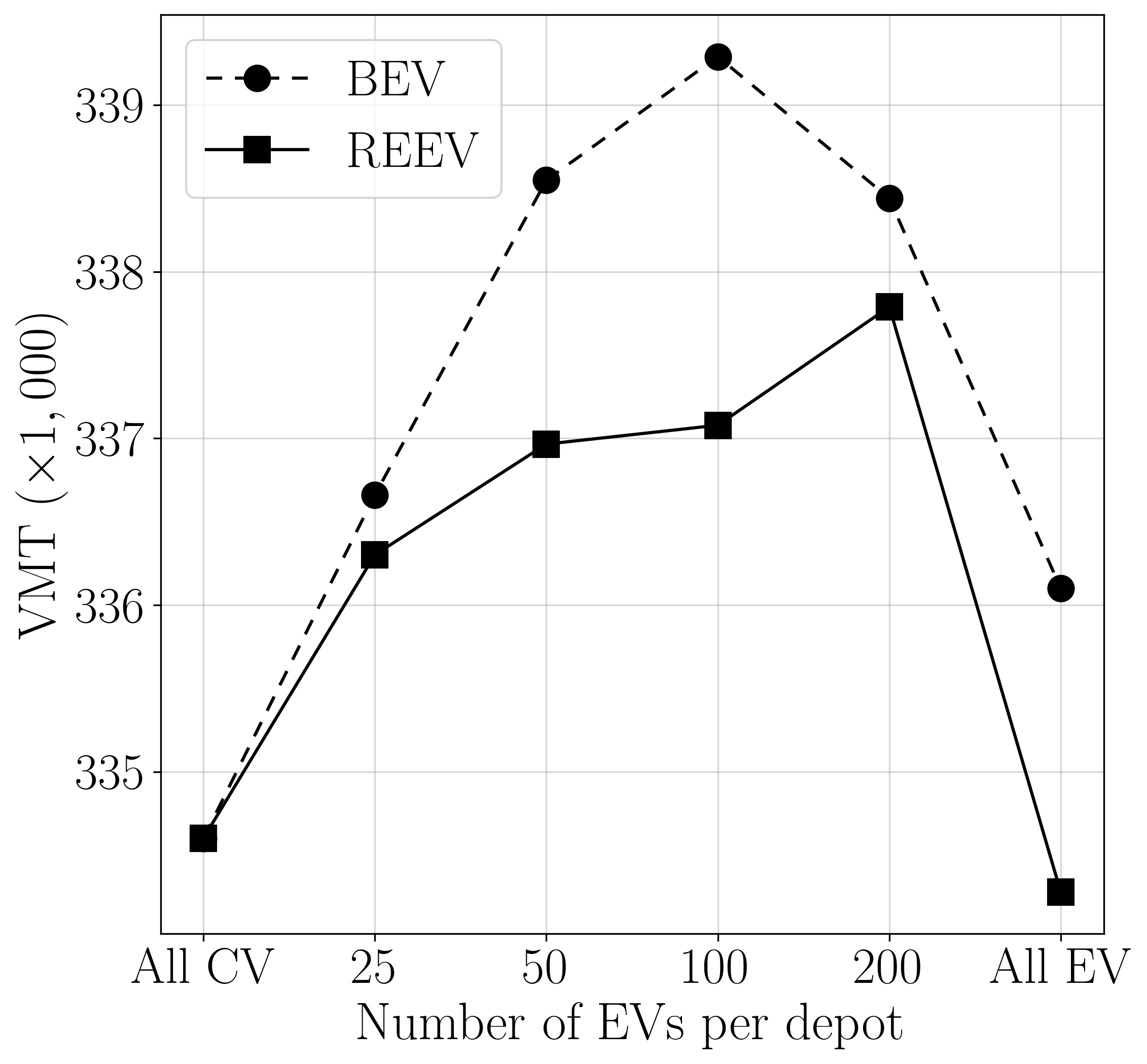}}
    \subfloat[VHT\label{fig:VHTgivendeployment}]{%
        \includegraphics*[width=0.34\textwidth,height=\textheight,keepaspectratio]{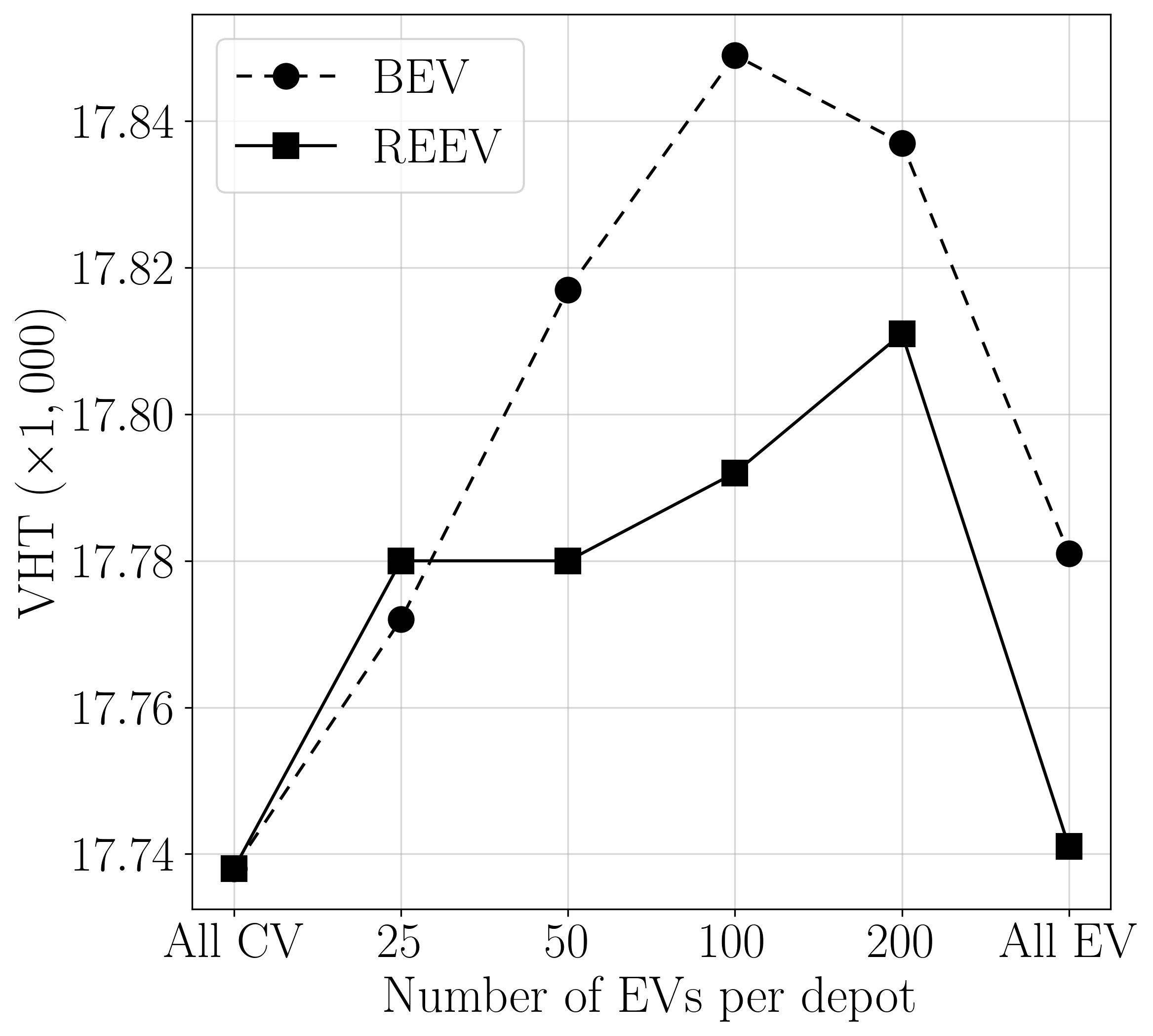}}\\
    \subfloat[VMT per mode in REEV case\label{fig:VMTpermodegivendeploymentREEV}]{%
        \includegraphics*[width=0.33\textwidth,height=\textheight,keepaspectratio]{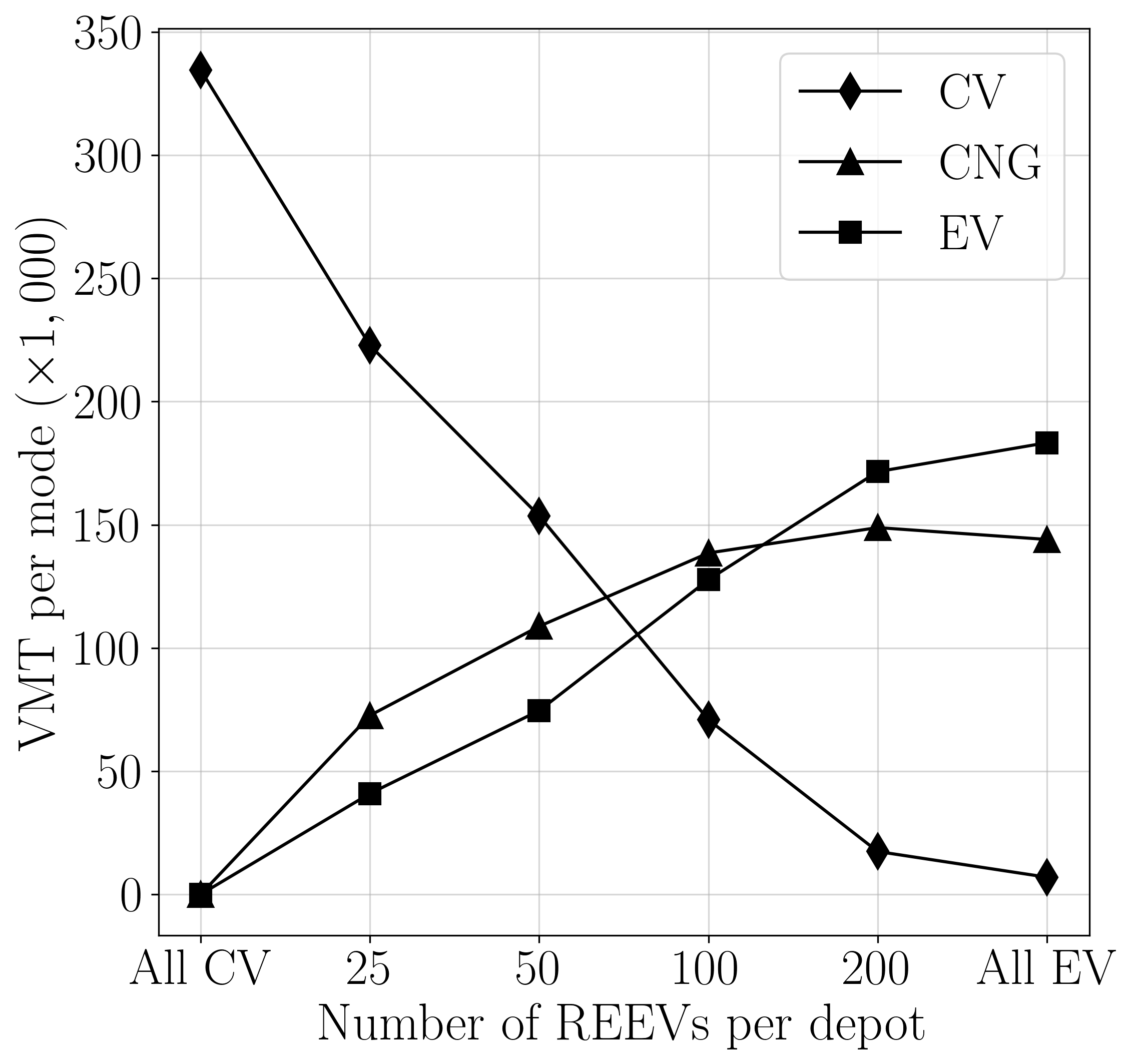}}
    \subfloat[VMT per mode in BEV case\label{fig:VMTpermodegivendeploymentBEV}]{%
        \includegraphics*[width=0.333\textwidth,height=\textheight,keepaspectratio]{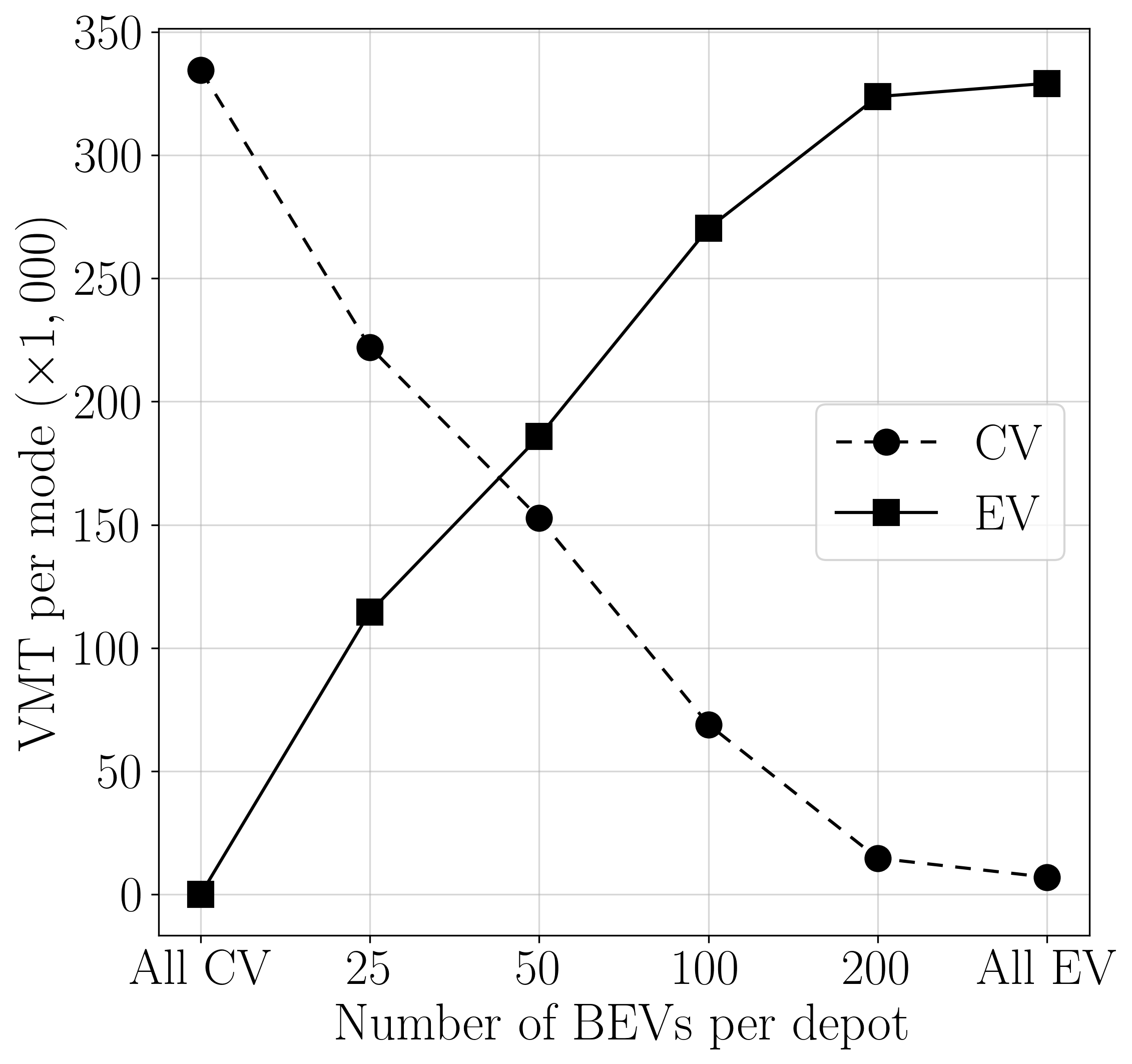}}
    \subfloat[Number of vehicles per type\label{fig:NumvehgivendeploymentpertypeREEV}]{%
        \includegraphics*[width=0.316\textwidth,height=\textheight,keepaspectratio]{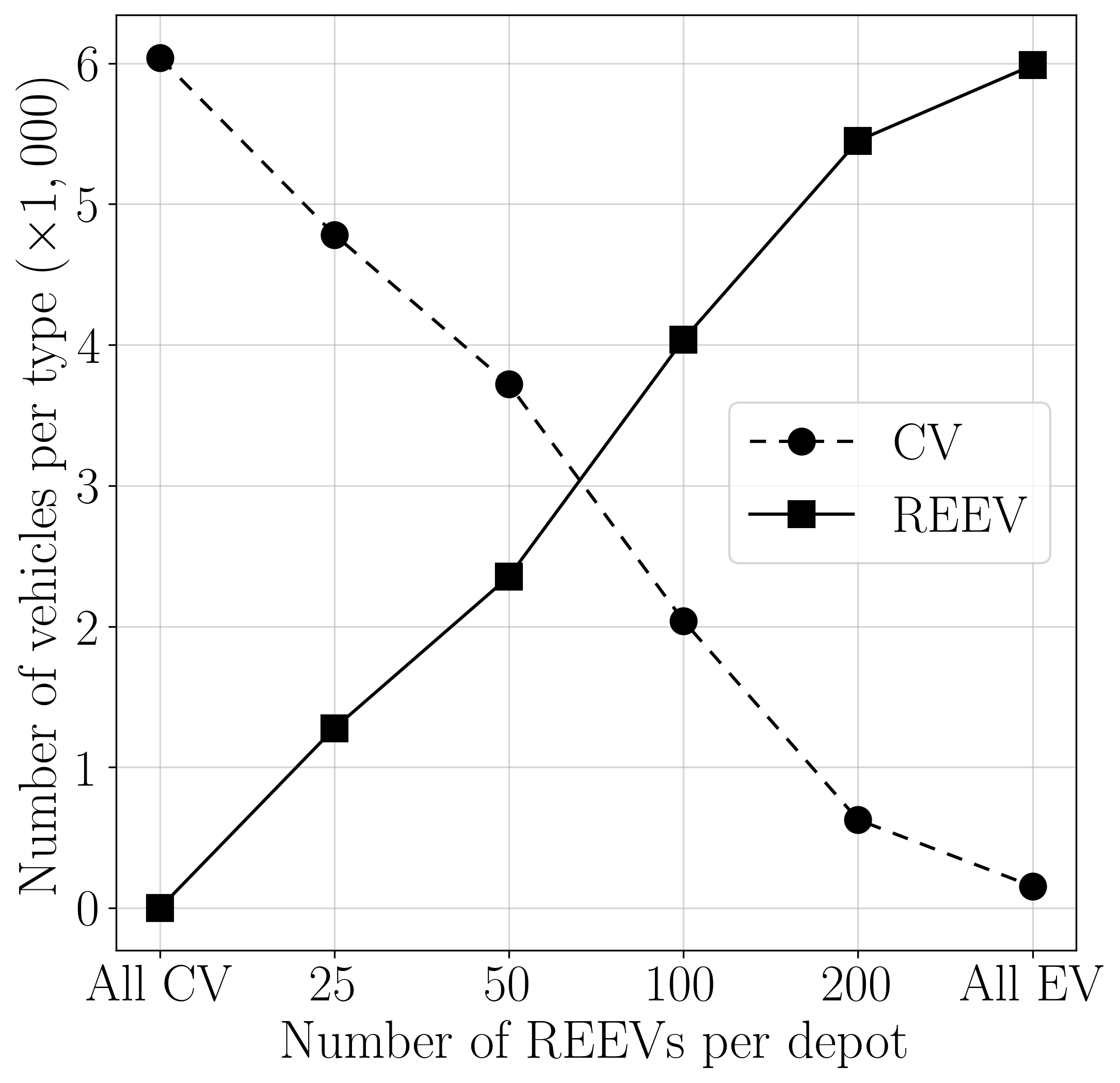}}
    \caption{Impact of EV deployment rate on cost, VMT, VHT, and fleet composition.} \label{fig:deploymentimpact_1}
\end{figure*}

\Cref{fig:deploymentimpact_1} summarizes the aggregate regional-level results.
First, \Cref{fig:Costgivendeployment} shows that the cost decrease relative to the All CV scenario, as a function of the number of EVs per depot.
We observe that using either REEVs or BEVs can substantially reduce operational costs by up to 67\%.
Perhaps more important, we see a difference of up to 11\% between the REEV and BEV cases in terms of cost reduction.
Since BEVs have higher infrastructure- and ownership-related fixed costs, REEVs may be preferable whenever the difference in fixed costs is higher than 11\%.
\Cref{fig:VMTgivendeployment} and \Cref{fig:VHTgivendeployment} report the regional-level VMT and VHT, respectively.
We observe that they exhibit a maximum when the number of EVs is roughly between 100 and 200, depending on the depot.
This is not unexpected since the All CV and All EV cases must travel roughly the same number of miles as a homogeneous fleet of vehicles and intermediate compositions must trade off between more EVs of cheaper travel cost and fewer CVs of higher costs.
This is supported %
by
\Cref{fig:VMTpermodegivendeploymentREEV} and \Cref{fig:VMTpermodegivendeploymentBEV},
which break down the VMT per drive mode in REEV and BEV cases, respectively.
We observe that as more EVs are deployed, the total EV and CNG miles increase but total CV miles decrease.
 \Cref{fig:NumvehgivendeploymentpertypeREEV} shows the breakdown of the overall regionwide fleet into REEVs and CVs.
Perhaps not surprising, more REEVs are utilized as they are made available.
The BEV case had the same trend with similar magnitudes, and we omit the corresponding figure.

The average regionwide total number of vehicles (regardless of type) across all scenarios was found to be 6,076 and was stable across scenarios.
The average number of packages delivered and VHT per vehicle were found to be 99.8 and 9.6 hours, respectively.
These statistics show that hours-of-work or duration limits were mostly binding whereas vehicle capacity utilization was roughly 83.2\% and hence mostly nonbinding.
The average VMT per vehicle was around 55.8 miles with a road travel speed of 35.8 mph, whereas the overall speed (VMT divided by VHT including service times) was 5.8 mph.
As a partial validation of our experimental design, we note that similar values have been reported for the average VMT (41.4 miles) of FedEx routes in the United States~\citep{barnitt2011fedex,feng2013economic}.
These statistics are crucial because they reveal potential areas of improvement on which manufacturers and policy makers can focus their attention.

\subsection{Impact of Extended Work Hours on Systemwide Performance Metrics}\label{sec:workhours}

By considering a baseline of 25 EVs per depot, we investigated the systemwide benefits of extended work hours in both the REEV and BEV cases. We focused on five scenarios, with $T$ ranging from 10 (set as a baseline) to 14 hours.
As we discussed in the preceding subsection, since vehicle routes are mostly constrained by their work hours (as opposed to their capacity), increasing driving time has an obvious positive impact on performance metrics. 

\Cref{fig:workhourimpact} illustrates the regionwide results.
First, \Cref{fig:Costgivenworkhours} shows a potential reduction in overall energy costs, of at least 17\% and ranging up to 32\%, compared with the All CV scenario.
Second, \Cref{fig:VMTgivenworkhours} and \Cref{fig:VHTgivenworkhours} illustrate that both VMT and VHT values decrease strongly as a function of increasing $T$.
Since we also observe a decline in the number of vehicles (see \Cref{fig:Numvehgivenworkhours}), it is evident that vehicle utilization increases with increasing $T$.
For example, when $T=14$ hours, the average number of deliveries increases to 115.9, which translates to roughly 96.6\% capacity utilization. \Cref{fig:VMTpermodegivenworkhoursREEV} and \Cref{fig:VMTpermodegivenworkhoursBEV} break down the VMT by drive mode.
We observe that as $T$ is increased, cheaper fuel options (i.e., electricity and CNG) are utilized slightly more while the expensive CV usage decreases.
Although not shown, we also observed that vehicle capacity constraints became binding for more routes as work hours were increased.

\begin{figure*}[!htb]
    \centering
    \subfloat[Cost\label{fig:Costgivenworkhours}]{%
        \includegraphics*[width=0.322\textwidth,height=\textheight,keepaspectratio]{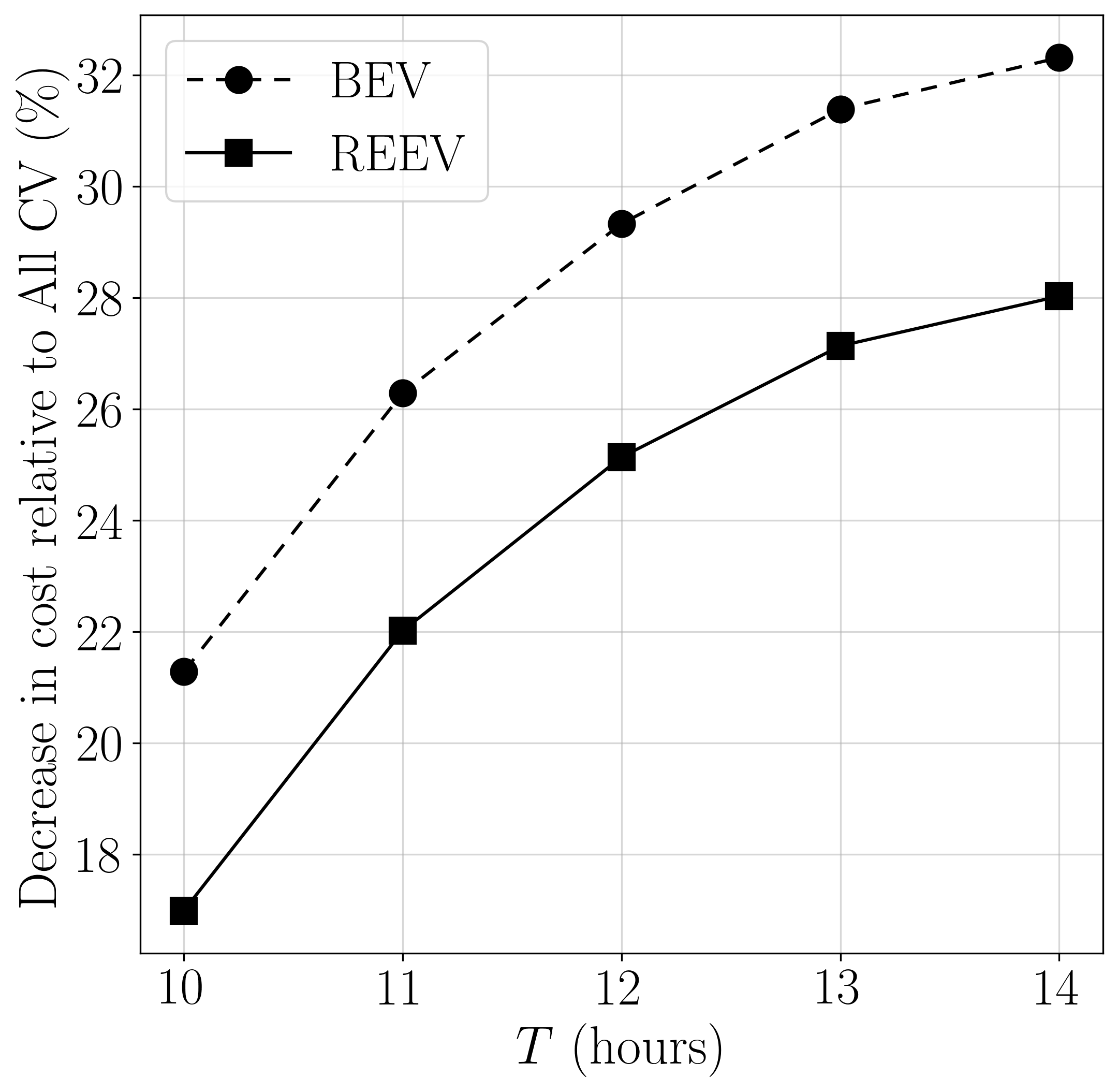}}
    \subfloat[VMT\label{fig:VMTgivenworkhours}]{%
        \includegraphics*[width=0.328\textwidth,height=\textheight,keepaspectratio]{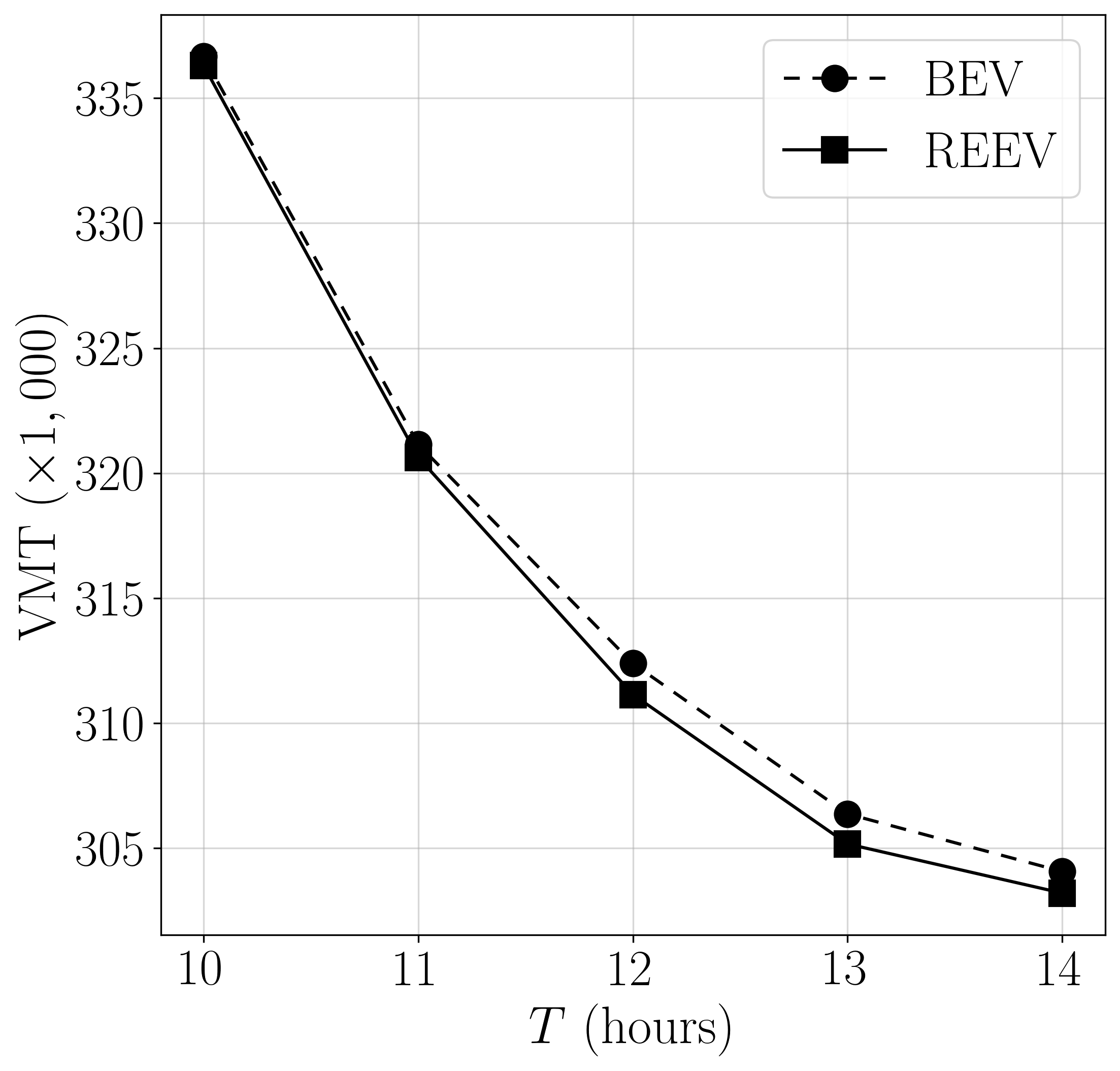}}
    \subfloat[VHT\label{fig:VHTgivenworkhours}]{%
        \includegraphics*[width=0.333\textwidth,height=\textheight,keepaspectratio]{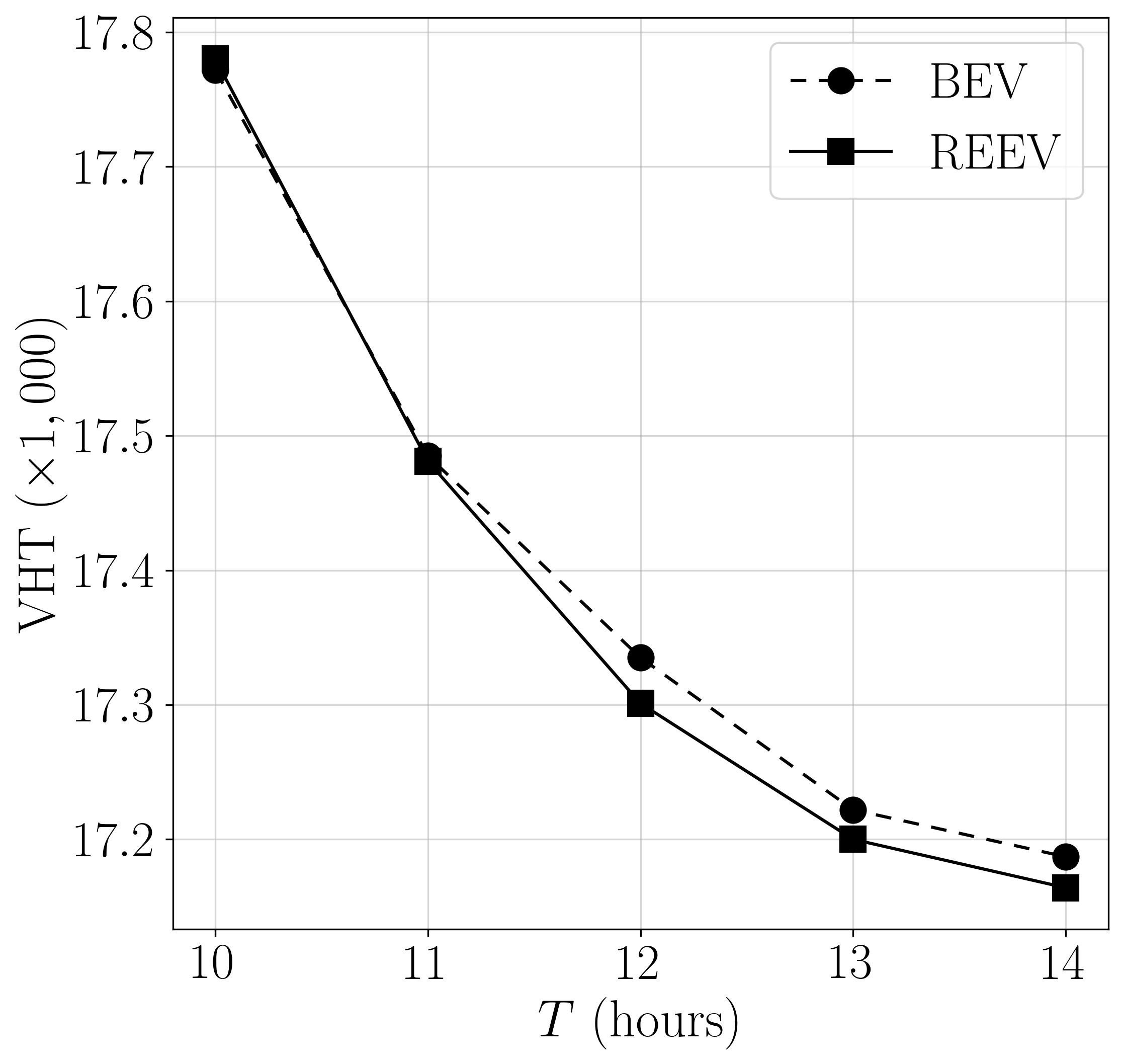}}\\
    \subfloat[VMT per mode in REEV case\label{fig:VMTpermodegivenworkhoursREEV}]{%
        \includegraphics*[width=0.33\textwidth,height=\textheight,keepaspectratio]{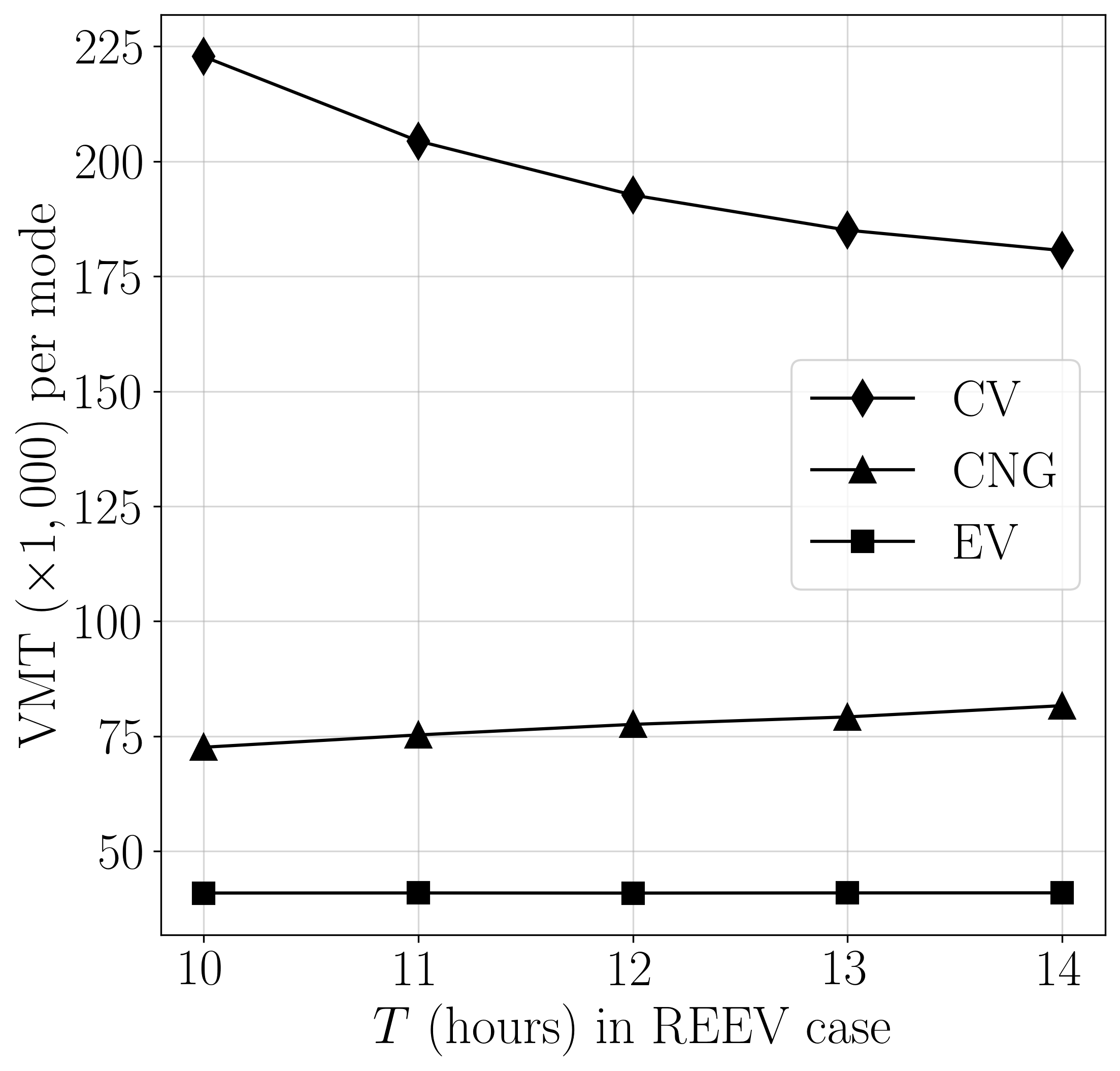}}
    \subfloat[VMT per mode in BEV case\label{fig:VMTpermodegivenworkhoursBEV}]{%
        \includegraphics*[width=0.33\textwidth,height=\textheight,keepaspectratio]{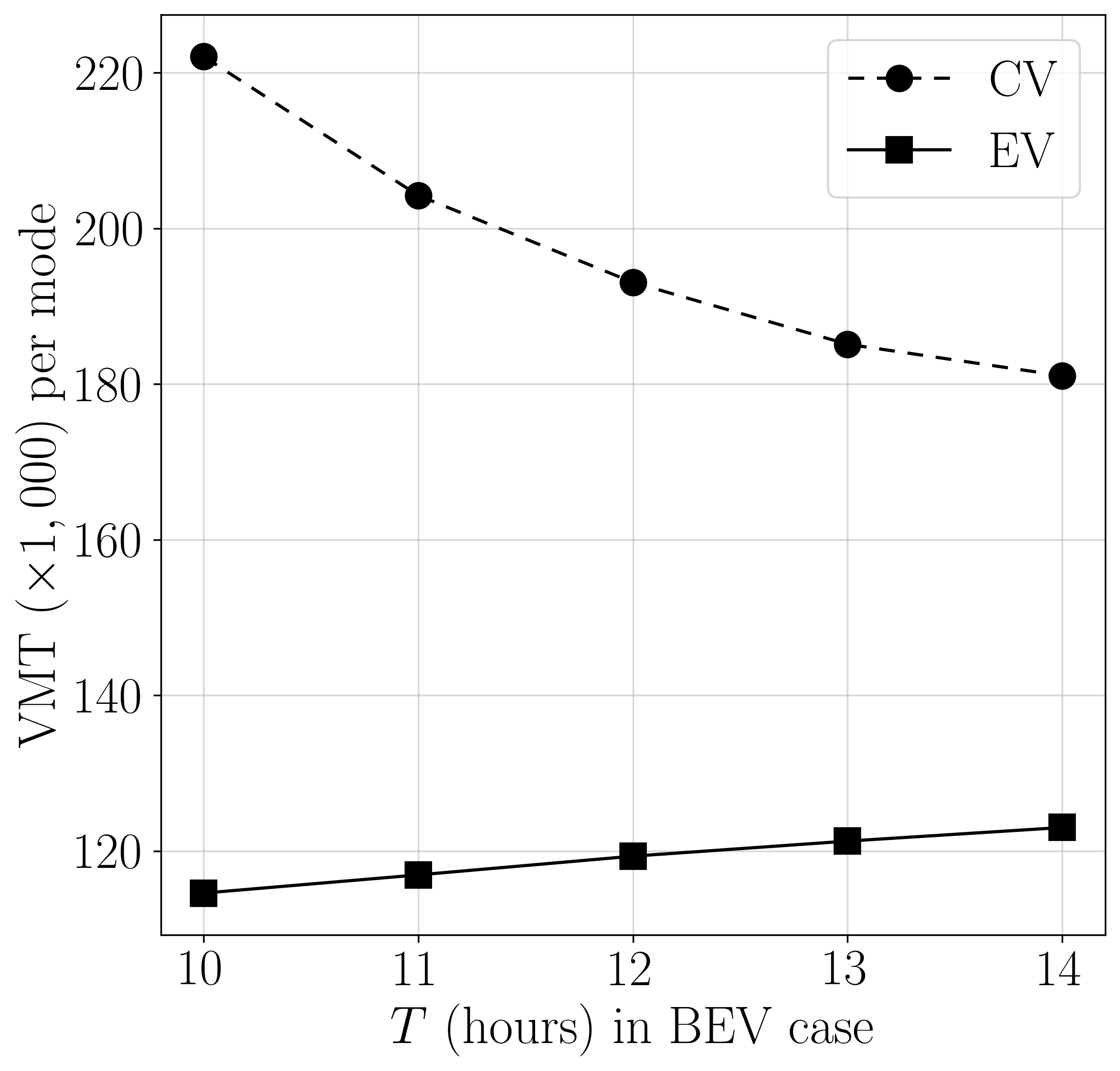}}
    \subfloat[Number of vehicles per type\label{fig:Numvehgivenworkhours}]{%
        \includegraphics*[width=0.327\textwidth,height=\textheight,keepaspectratio]{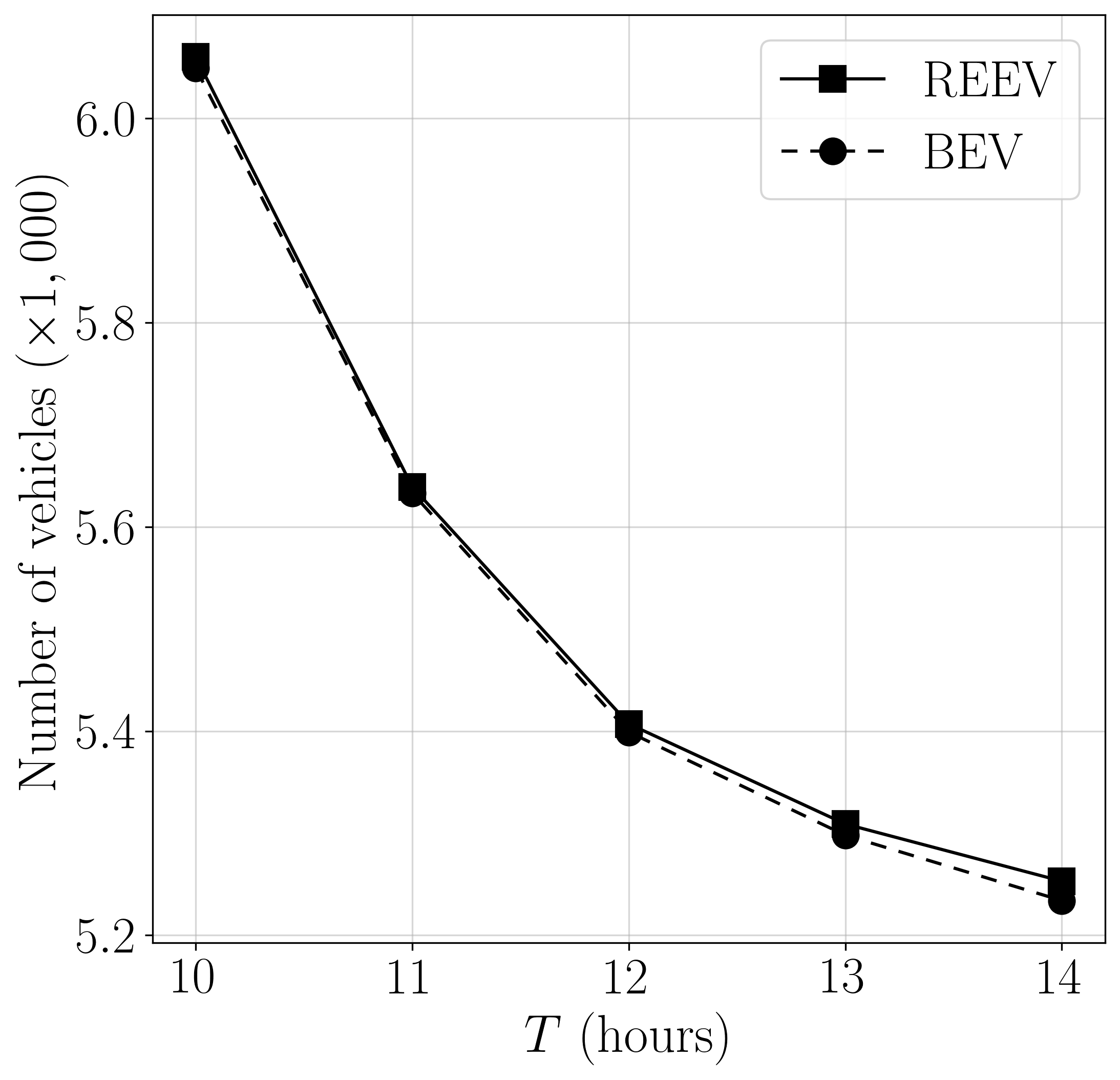}}
    \caption{Impact of extended work hours on cost, VMT, VHT, and fleet composition.} \label{fig:workhourimpact}
\end{figure*}

\subsection{Impact of EV Range on Systemwide Performance Metrics}\label{sec:EVrange}

Here we  investigate the impact of increasing the EV range in both the REEV and BEV cases.
In the REEV case we consider $D_E \in \{33, 40, 60, 80, 100, 125, 150\}$ miles, whereas in the BEV case we consider $D_E\in\{100, 125, 150, 200, 250\}$ miles, where the lower limit of 100 miles is inspired by recent news reports of \citet{Abt,Kane}. 

\Cref{fig:rangeimpact} summarizes the relevant findings.
\Cref{fig:Costgivenrange} reports cost decreases relative to the All CV scenario.
We observe that the cost decrease plateaus in the BEV case after $D_E = 200$ miles. %
This indicates that  very few vehicle routes  require traveling more than 200 miles.
Moreover, \Cref{fig:VMTpermodegivenrangeREEV} and \Cref{fig:VMTpermodegivenrangeBEV} show that the total VMT (summed across conventional and EV modes) do not change with increasing $D_E$; indeed, the average VMT per vehicle remains 55.8 miles across all $D_E$ values, similar to \Cref{sec:deployment}.
Therefore, the cost decrease cannot be attributed to a decrease in VMT but rather to a shift from fossil fuel to electricity.
This highlights a key takeaway: Although e-commerce LSPs can expect operational cost reductions of roughly 18\% by adopting REEVs with an EV range of up to $D_E = 50$ miles, doubling this range further to (very optimistic) values beyond 100 miles will lead to further reductions of at most 5\%.
Indeed, the average VMT of roughly $56$ miles can serve as a good design principle to guide investment in more efficient hybrid powertrains.
We note that it is difficult to draw a fair comparison between the REEV and BEV cases, given significant differences in their powertrain technologies and EV ranges, as well as fixed investment and ownership-related costs.

\begin{figure*}[!htb]
    \centering
    \subfloat[Cost.\label{fig:Costgivenrange}]{%
        \includegraphics*[width=0.3\textwidth,height=\textheight,keepaspectratio]{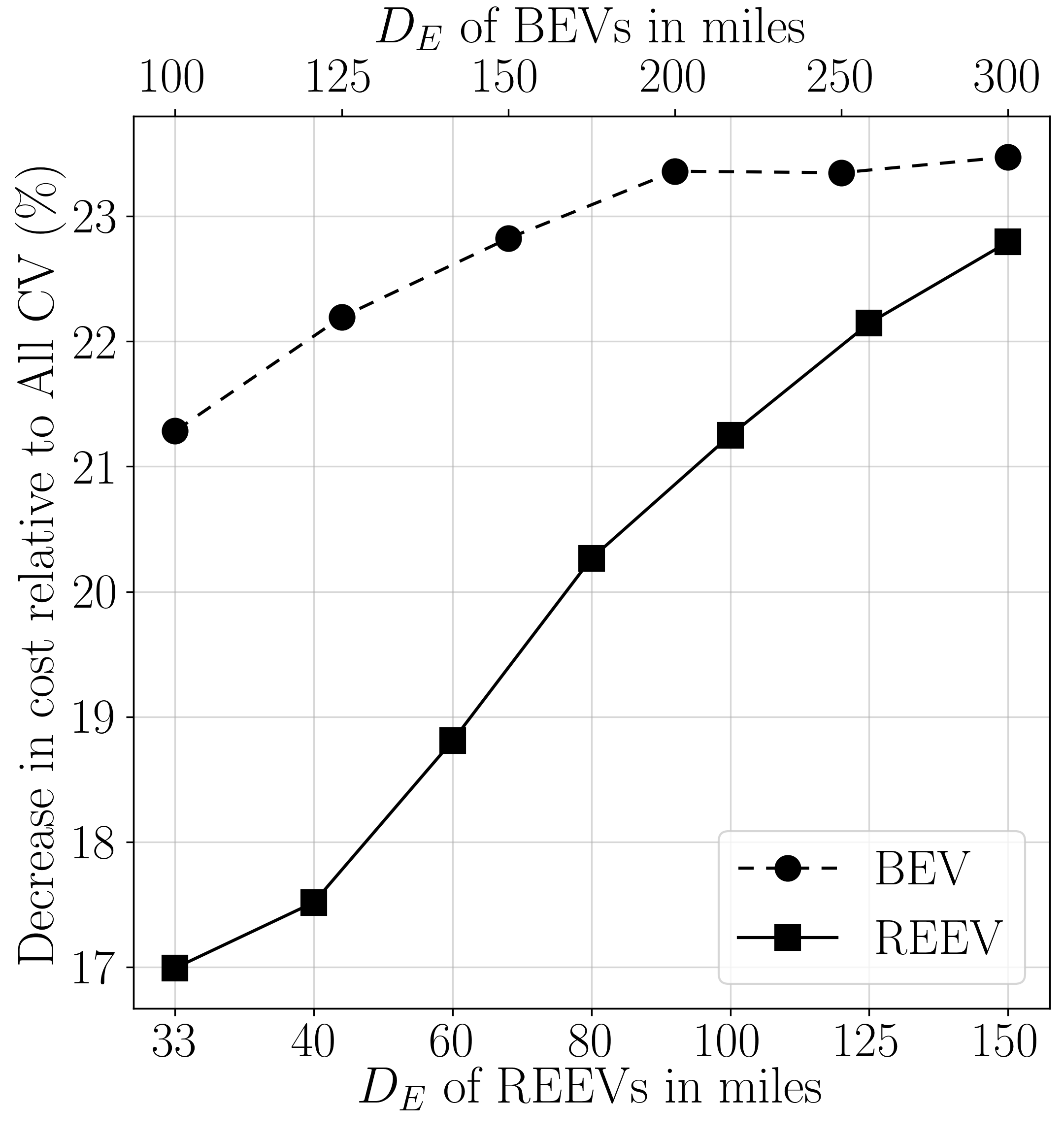}}
    \subfloat[VMT per mode in REEVs.\label{fig:VMTpermodegivenrangeREEV}]{%
        \includegraphics*[width=0.31\textwidth,height=\textheight,keepaspectratio]{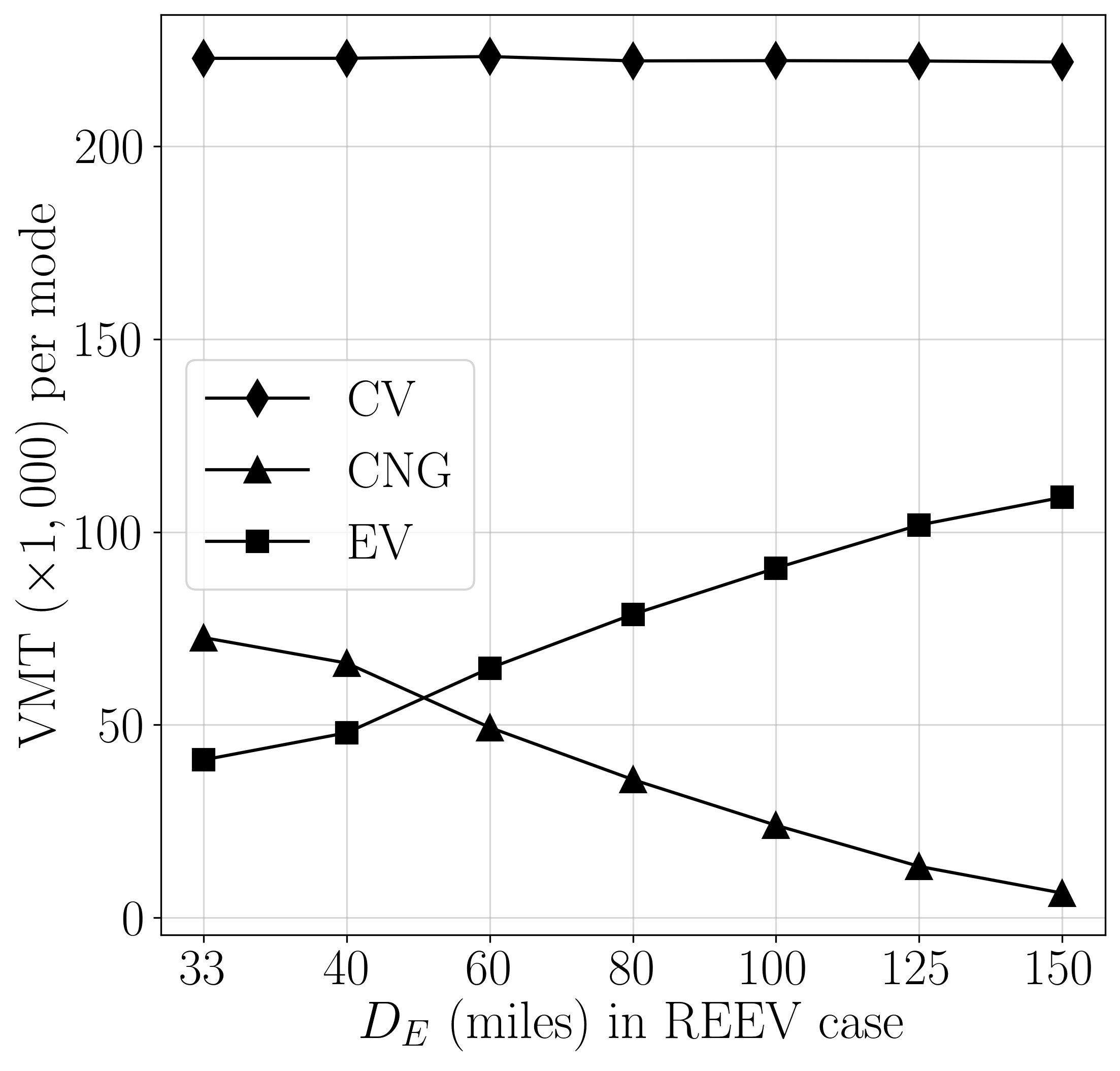}}
    \subfloat[
    VMT per mode in BEVs.\label{fig:VMTpermodegivenrangeBEV}]{%
        \includegraphics*[width=0.37\textwidth,height=\textheight,keepaspectratio]{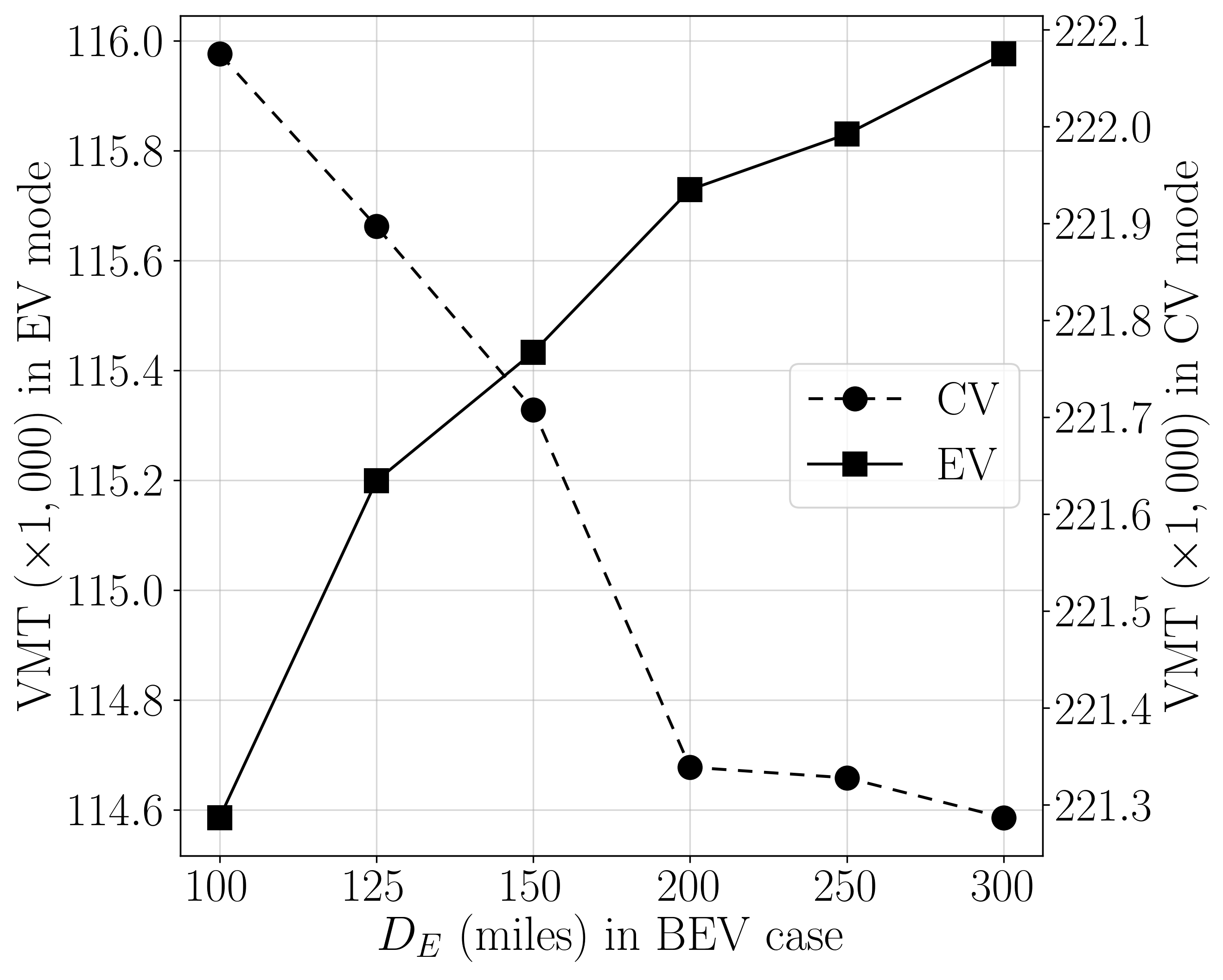}}
    \caption{Impact of extended EV range on cost and VMT.} \label{fig:rangeimpact}
\end{figure*}

\subsection{Impact of Vehicle Capacities on Systemwide Performance Metrics}\label{sec:capacity}

Compared with the baseline value, we allowed higher vehicle capacities of $Q\in\{150, 180, 210, 240\}$ packages in both the REEV and BEV cases.
However, these did not have any significant impacts.
Indeed, as noted in \Cref{sec:deployment}, the average number of packages delivered was roughly 99.8.
Therefore, increasing vehicle capacities beyond $120$ packages is not expected to yield any benefits.
Nonetheless, since the routes are mostly time-limited, it is possible that increasing vehicle capacity while simultaneously reducing the average service time to $\sigma = 2$ minutes (say) could potentially yield cost savings.
We investigate this case  next.

\subsection{Impact of Service Time on Systemwide Performance mMtrics}\label{sec:service_time}

The choice of the service time is an important parameter that dictates the average route lengths and the number of customers served.
We investigate $\sigma\in\{0, 1, 2, 3, 4, 5\}$ minutes focusing only on the REEV case.
Note that very low service times are difficult to achieve in practice;  we consider them only to establish theoretical upper bounds on savings.
\Cref{fig:service_time} summarizes the performance metrics results for these scenarios, relative to a baseline of $\sigma=5$ minutes.
\Cref{fig:Costgivensigma} and \Cref{fig:Numvehgivensigma} show that service time alone contributes to roughly 26\% of the overall costs and 29.2\% of the fleet size in the baseline scenario (compare with $\sigma = 0$).
\Cref{fig:VMTgivensigma} and \Cref{fig:VHTgivensigma} indicate that VMT and VHT are also significantly impacted by service times.
Whereas the total VHT increases roughly linearly with $\sigma$, \Cref{fig:VHTonmajorgivensigma} shows that the VHT on major roads (which we define as travel between superlocations) increases almost exponentially with $\sigma$.
 \Cref{fig:Numvehpertypegivensigma} shows that the number of REEVs remains fairly constant across the range of $\sigma$ values (since all available REEVs are used in each scenario), whereas the number of CVs increases sharply to respond to service time increases.

\begin{figure*}[!htb]
    \centering
    \subfloat[Cost\label{fig:Costgivensigma}]{%
        \includegraphics*[width=0.33\textwidth,height=\textheight,keepaspectratio]{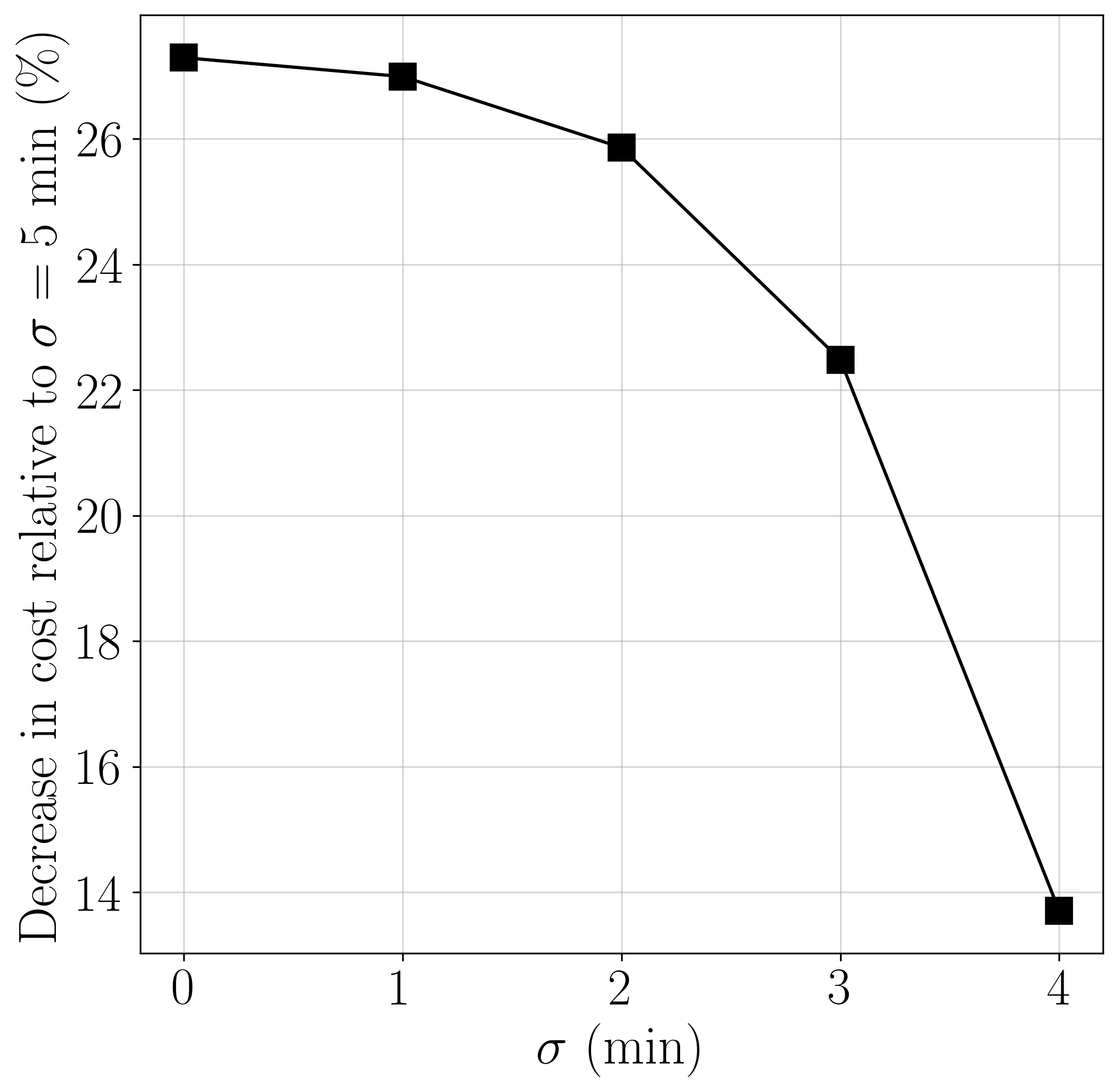}}
    \subfloat[VMT\label{fig:VMTgivensigma}]{%
        \includegraphics*[width=0.337\textwidth,height=\textheight,keepaspectratio]{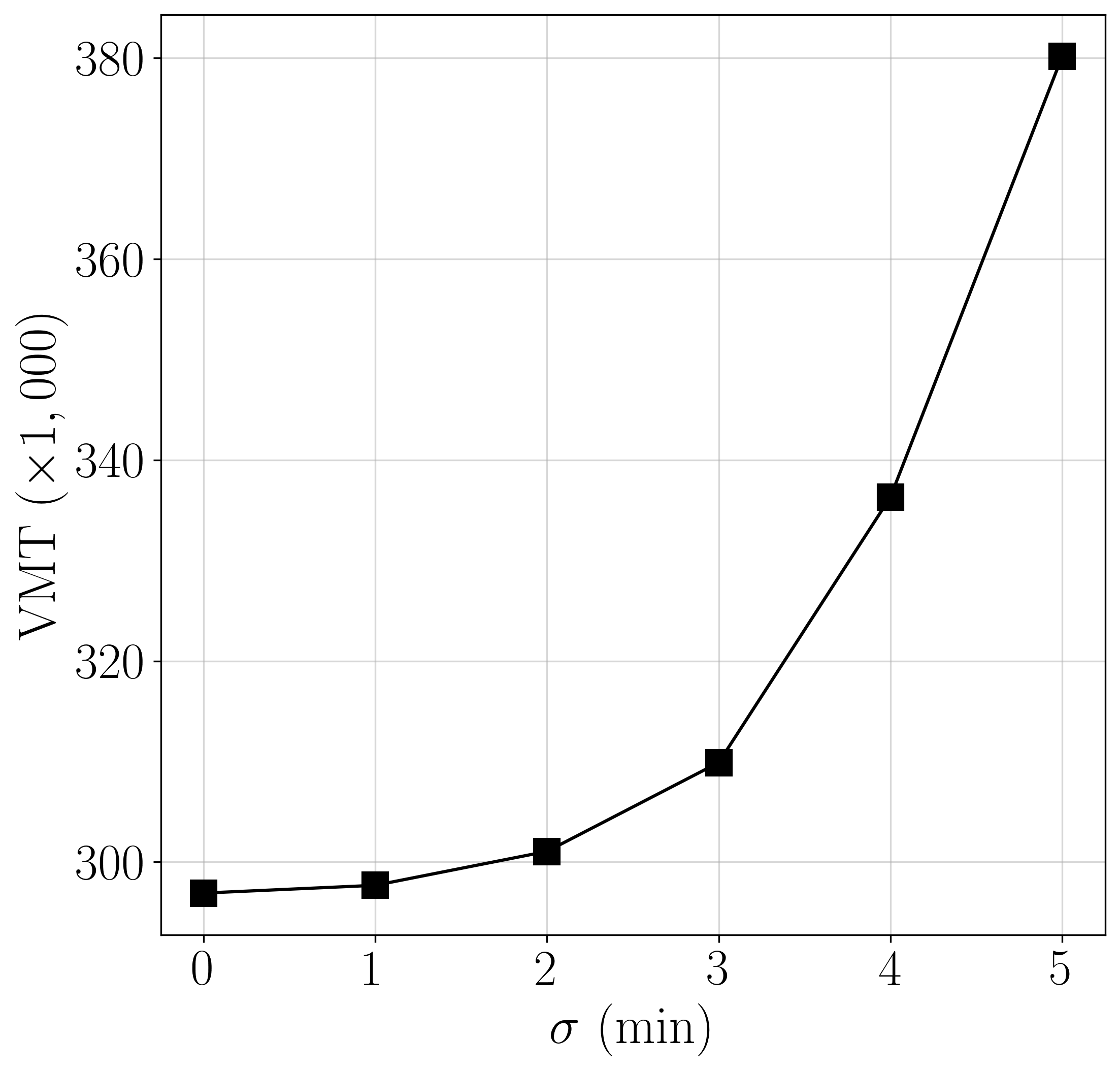}}
    \subfloat[VHT\label{fig:VHTgivensigma}]{%
        \includegraphics*[width=0.328\textwidth,height=\textheight,keepaspectratio]{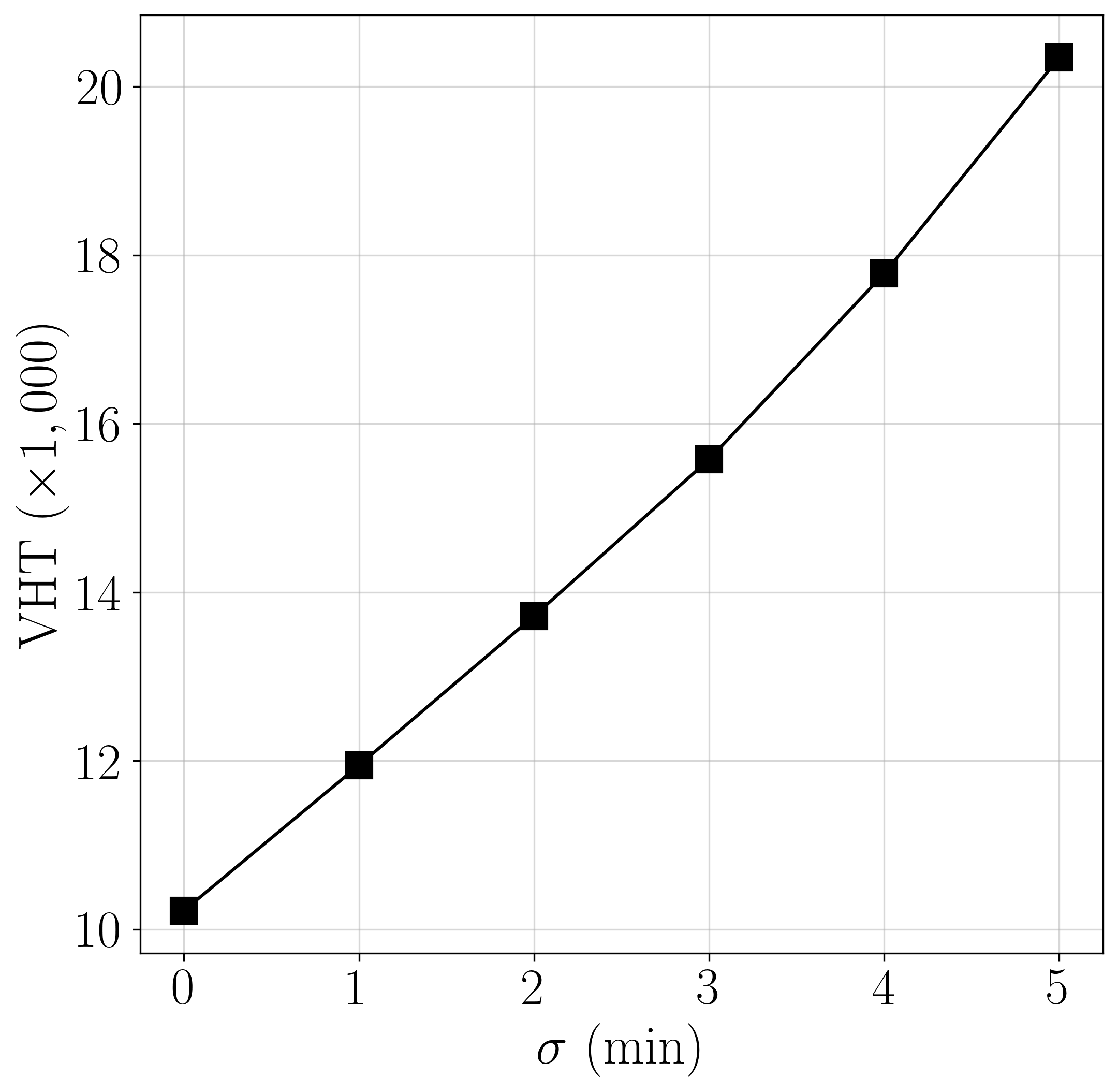}}\\
    \subfloat[VHT (major roads only)\label{fig:VHTonmajorgivensigma}]{%
        \includegraphics*[width=0.33\textwidth,height=\textheight,keepaspectratio]{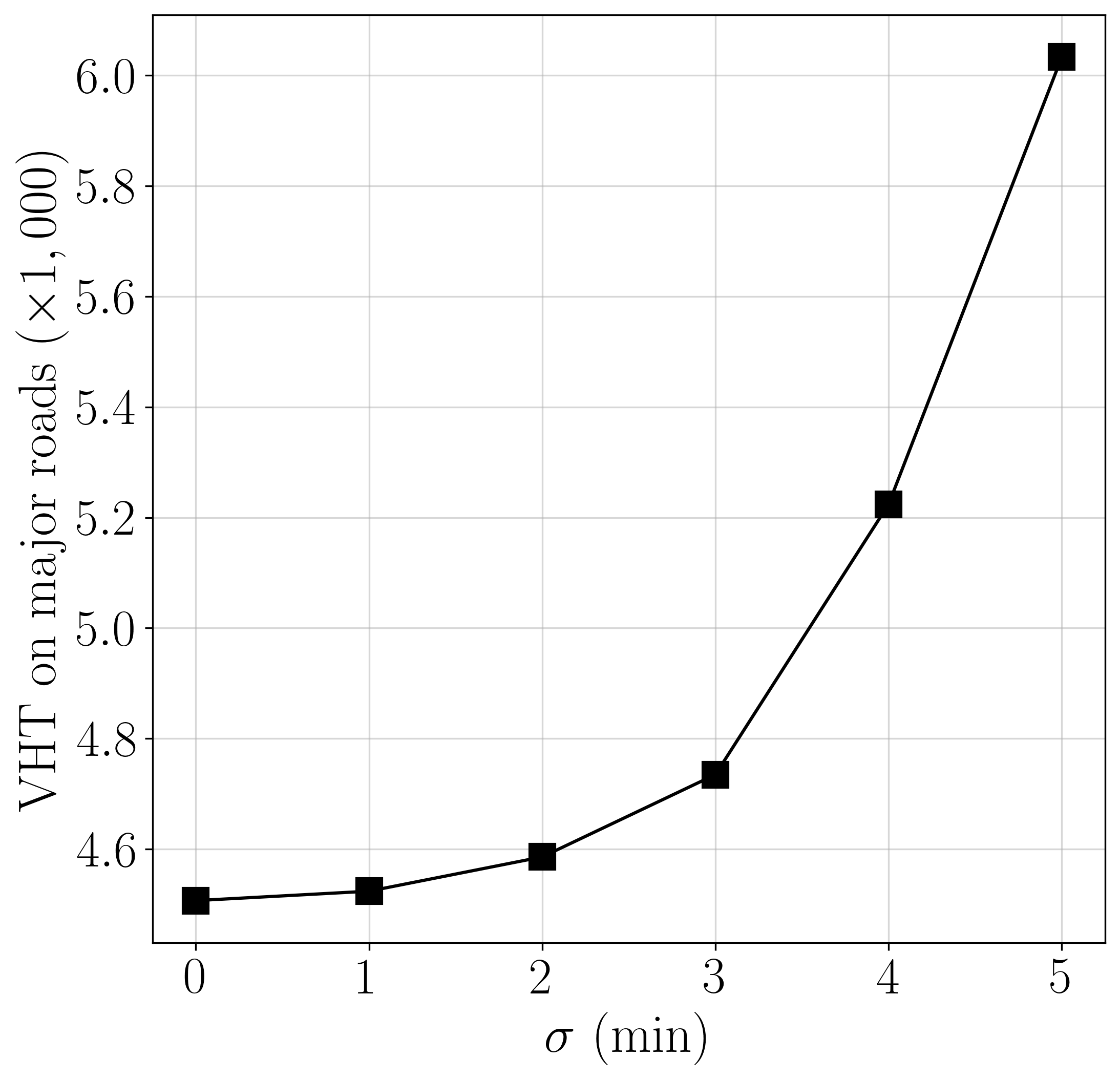}}
    \subfloat[Number of vehicles\label{fig:Numvehgivensigma}]{%
        \includegraphics*[width=0.33\textwidth,height=\textheight,keepaspectratio]{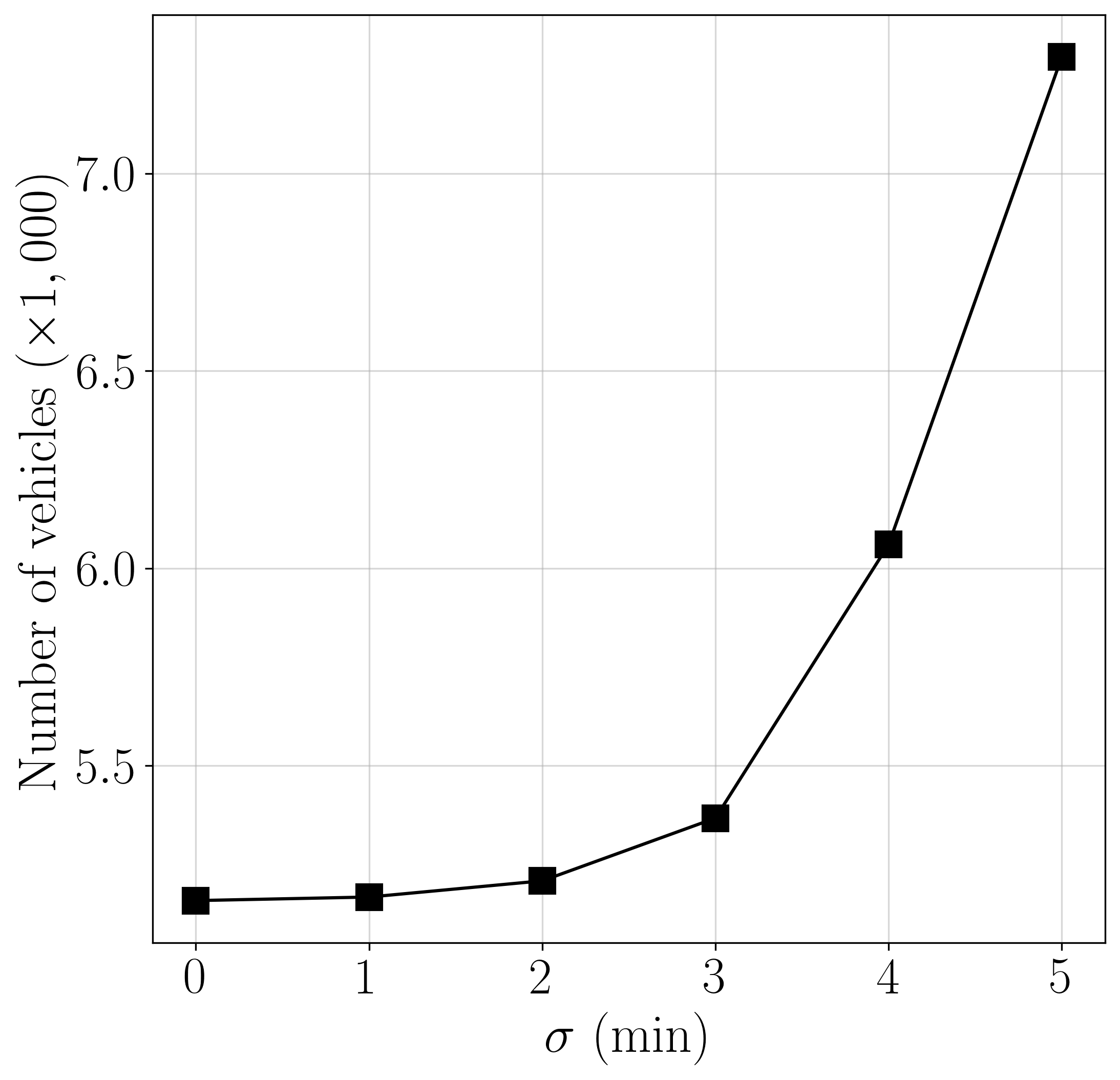}}
    \subfloat[Number of vehicles per type\label{fig:Numvehpertypegivensigma}]{%
        \includegraphics*[width=0.327\textwidth,height=\textheight,keepaspectratio]{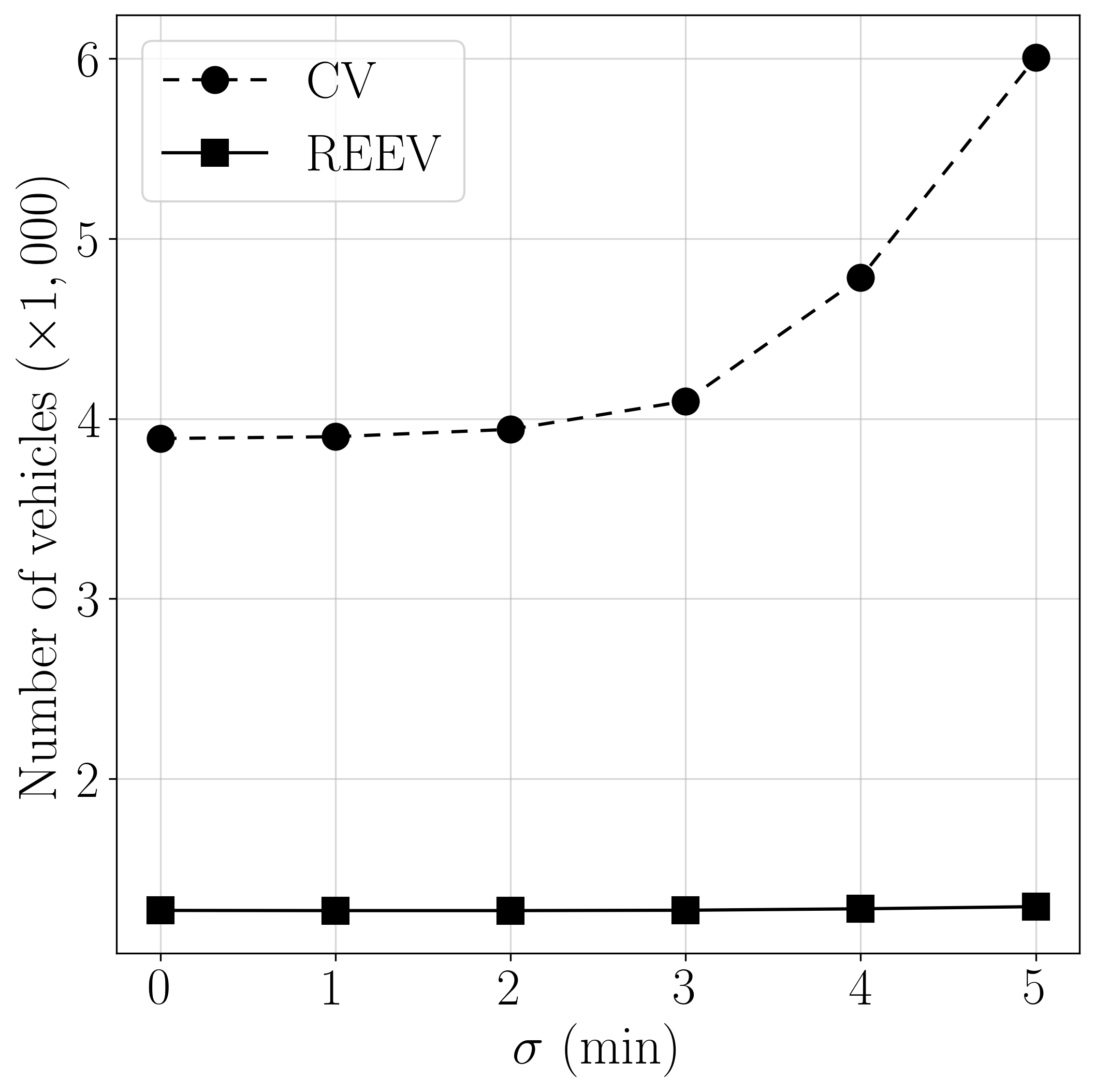}}
    \caption{Impact of service time on cost, VMT, VHT, and fleet composition.} \label{fig:service_time}
\end{figure*}

\section{Conclusions}\label{sec:conclusions}

Range extenders have an important role to play in striking the right balance between the federal government's goal of increasing EV adoption for e-commerce distribution on the one hand and logistics service providers' requirements of achieving operational cost savings on the other.
The reason is that by extending the driving range of traditional battery-operated vehicles, they combine the higher efficiency and environmental benefits provided by the latter with the longer range and autonomy of conventional fossil-fuel-powered vehicles.
This combination makes them particularly suited for e-commerce deliveries in dense urban areas, where battery recharging along routes can be too time-consuming to economically justify the use of all-electric vehicles.

This paper investigated the impacts of adopting REEV technology in a large-scale real-world case study concerning parcel deliveries in the Chicago metropolitan area.
With a goal to quantify various operational metrics including energy costs, vehicle miles traveled in electric and fossil fuel modes, vehicle hours traveled, and their capacity utilization, we defined a new VRP varian---called the range-extended electric vehicle routing problem, or REEEVRP---that simultaneously determines the optimal fleet composition and truck routes for a mixed fleet of REEVs and CVs.
Unlike other heterogeneous VRP variants, the REEVRP features heterogeneity in alternative propulsion modes within the same vehicle type.
To solve the REEVRP, we proposed an exact BPC algorithm and an ITS metaheuristic.
The latter is capable of providing solutions for the large-scale instances in our study featuring  more than 3,000 delivery locations, whereas the former provides near-optimal solutions (roughly 5\% on average) for small-scale instances.
We found that distributors can expect energy cost savings of up to 17\% while incurring less than 0.5\% increase in VMT and VHT by deploying only a small number of REEVs with a modest EV range of 33 miles.
Moreover, since the routes are time-limited, doubling the EV range to 60 miles further reduced costs by only 4\%, something that can also be achieved by decreasing the average service time by 1~minute or increasing driver working time by 1~hour.
Our case study also revealed other potential areas of improvement to which manufacturers, logistics service providers, and policy makers can focus their attention.

Future directions of study include
the design of a more efficient exact BPC algorithm (for example, where the limited EV range of REEVs are directly handled as part of the pricing routine); 
the modeling of time-dependent travel times and intraroute power management decisions;
and sensitivity analyses of parameters that our work did not consider, such as cost coefficients as well as neighborhood-level travel times and distances.

\section*{Acknowledgments}
This material is based upon work supported by the U.S.~Department of Energy,
Office of Science, under contract number DE-AC02-06CH11357.

\bibliographystyle{unsrtnat}
\bibliography{refs}

\vfill
\center{
\framebox{\parbox{0.98\linewidth}{
\footnotesize
The submitted manuscript has been created by UChicago Argonne, LLC, Operator of 
Argonne National Laboratory (``Argonne''). Argonne, a U.S.\ Department of 
Energy 
Office of Science laboratory, is operated under Contract No.\ 
DE-AC02-06CH11357. 
The U.S.\ Government retains for itself, and others acting on its behalf, a 
paid-up nonexclusive, irrevocable worldwide license in said article to 
reproduce, prepare derivative works, distribute copies to the public, and 
perform publicly and display publicly, by or on behalf of the Government.  The 
Department of Energy will provide public access to these results of federally 
sponsored research in accordance with the DOE Public Access Plan. 
http://energy.gov/downloads/doe-public-access-plan.}}}

\end{document}